\documentclass{amsart}
\usepackage{amsmath,amssymb,verbatim,color,enumerate,graphicx,caption}

\usepackage{mathtools}

\usepackage{graphicx}
\usepackage{amsmath}
\usepackage{amsfonts,amssymb}
\usepackage{amsthm}
\usepackage{verbatim}
\usepackage{array}
\usepackage{graphicx,color}
\usepackage[normalem]{ulem}

\usepackage[bookmarks=false,unicode,colorlinks,urlcolor=red,citecolor=red,linkcolor=blue]{hyperref}

\newcommand{\bo}{{\rm O}}
\newcommand{\so}{{\rm o}}
\newcommand{\ds}{\displaystyle}
\newcommand{\dsum}{\ds\sum}
\newcommand{\dint}{\ds\int}
\newcommand{\eqskip}{ \vspace*{2mm}\\ }
\newcommand{\fr}[2]{\frac{\ds #1}{\ds #2}}
\newcommand{\R}{\mathbb{R}}

\numberwithin{equation}{section}

\newcommand\numberthis{\addtocounter{equation}{1}\tag{\theequation}}

\newtheorem{problem}[]{Problem}
\theoremstyle{plain}
\newtheorem{theorem}{Theorem}[section]

\theoremstyle{remark}
\newtheorem{remark}[theorem]{Remark}

\newcommand{\ap}{\alpha_{+}}
\newcommand{\am}{\alpha_{-}}

\title[Clamped plates under large compression]{On the behaviour of clamped plates under large compression}
\author{P.\ R.\ S.\ Antunes, D.\ Buoso \and P.\ Freitas}
\address{
Grupo de F\'{\i}sica Matem\'{a}tica, Faculdade de Ci\^encias, Universidade de Lisboa,
Campo Grande, Edif\'icio C6, P-1749-016 Lisboa, Portugal}
\email{prantunes@fc.ul.pt}

\address{
EPFL, SB MATH SCI-SB-JS, Station 8, CH-1015 Lausanne, Switzerland}
\email{davide.buoso@epfl.ch}

\address{Departamento de Matem\'atica, Instituto Superior T\'ecnico, Universidade de Lisboa, Av. Rovisco Pais 1,
P-1049-001 Lisboa, Portugal {\rm and}
Grupo de F\'isica Matem\'{a}tica, Faculdade de Ci\^encias, Universidade de Lisboa,
Campo Grande, Edif\'icio C6, P-1749-016 Lisboa, Portugal}
\email{psfreitas@fc.ul.pt}

\begin{document}

\begin{abstract}
We determine the asymptotic behaviour of eigenvalues of clamped plates under large compression, by relating this problem
to eigenvalues of the Laplacian with Robin boundary conditions. Using the method of fundamental solutions, we then carry out
a numerical study of the extremal domains for the first eigenvalue, from which we see that these depend on the
value of the compression, and start developing a boundary structure as this parameter is increased. The corresponding
number of nodal domains of the first eigenfunction of the extremal domain also increases with the compression.
\end{abstract}

\maketitle

\section{Introduction}

Let $\Omega$ be a smooth bounded domain in $\R^N$, $N\ge2$.
We are interested in the following eigenvalue problem
\begin{equation}
\label{dirichlet}
\left\{\begin{array}{ll}
\Delta^2u+\alpha\Delta u=\lambda u, & {\rm \ in\ }\Omega,\eqskip
u=\fr{\partial u}{\partial \nu}=0, & {\rm \ on\ }\partial \Omega,
\end{array}\right.
\end{equation}
considered as a model for a clamped plate. Here $\alpha$ is a real parameter corresponding to the quotient between the tension and the flexural rigidity and,
depending on its sign, represents whether the plate is under tension ($\alpha<0$) or compression ($\alpha > 0$).
For domains $\Omega$ as described above, the eigenvalues of~\eqref{dirichlet} form an infinite sequence
\[
 \lambda_{1}\leq \lambda_{2} \leq \dots \leq \lambda_{k} \leq \dots,
\]
where $\lambda_{k}=\lambda_{k}(\Omega,\alpha)$ approaches $+\infty$ as $k$ goes to infinity.

The study of this and similar problems has been considered in the literature continuously over time since the works of Lord Rayleigh \cite{rayleigh} and Love \cite{love} on
clamped plates. We refer the reader to the book \cite{ggs} for an extensive historical and scientific overview on the mechanics of plates through the Kirchhoff-Love model,
which leads to problem \eqref{dirichlet}.

In this paper, we are concerned with two issues related to~\eqref{dirichlet}, namely, the asymptotic behaviour of the eigenvalues $\lambda_{k}$ as the
parameter $\alpha$ approaches $+\infty$ (the case of $-\infty$ was considered in~\cite{frank}), and the extremal domains of such eigenvalues as $\alpha$ varies.
In the first instance, the above problem is closely related to
\begin{equation}
\label{dirichlet2}
\left\{\begin{array}{ll}
\Delta^2v+a v + \gamma \Delta v=0, & {\rm \ in\ }\Omega,\eqskip
v=\fr{\partial v}{\partial \nu}=0, & {\rm \ on\ }\partial \Omega,
\end{array}\right.
\end{equation}
where now the eigenvalue parameter is $\gamma=\gamma(a)$, and the (positive) parameter $a$ stands for the elasticity constant of the medium surrounding
the plate. We know from a result in~\cite{kawlevvel} that for~\eqref{dirichlet2}
\[
 \lim_{a\to +\infty} \fr{\gamma_{1}(a)}{\sqrt{a}} = 2,
\]
which when translated into the eigenvalue problem in~\eqref{dirichlet} yields
\[
 \lim_{\alpha\to +\infty} \fr{\lambda_{1}(\alpha)}{\alpha^2} = -\fr{1}{4}.
\]
Our main result along these lines is to extend this to all eigenvalues $\lambda_{k}$. This is achieved by a different approach from that
used in~\cite{kawlevvel}, involving now a connection which, to the best of our knowledge, is new, between the eigenvalues of
the clamped plate problem~\eqref{dirichlet} and
those of a Robin eigenvalue problem for the Dirichlet Laplacian in the case where $\Omega$ is a ball -- see Section~\ref{robinrelation} for the details.
To be more precise, we prove the following
\begin{theorem}[Asymptotic behaviour of the $k^{\rm th}$ eigenvalue]\label{mainthm}
Let $\Omega$ be a bounded domain in $\R^N$. Then, for any positive integer $k$, the eigenvalues of~\eqref{dirichlet} satisfy
\begin{equation}
\label{gendom}
\lambda_k(\Omega,\alpha)=-\fr{\alpha^2}{4} + \so(\alpha^2),
\end{equation}
as $\alpha\to+\infty$. Moreover,
\begin{equation}
\label{gendom1}
\lambda_1(\Omega,\alpha)= -\fr{\alpha^2}{4}+\bo(\alpha),
\end{equation}
as $\alpha\to+\infty$.
\end{theorem}
For positive values of $\alpha$, each of the eigenvalue curves $\lambda_{k}=\lambda_{k}(\alpha)$ is, in fact, made up of analytic eigenvalue branches 
which intersect each other - see Figure~\ref{fig:figurasde}, where to illustrate this effect we plotted the quantity $\lambda_k(\Omega,\alpha)+\fr{\alpha^2}{4}$ for
the disk and for ellipses.
\begin{figure}[h!]
\centering
(a) \includegraphics[width=0.47\textwidth]{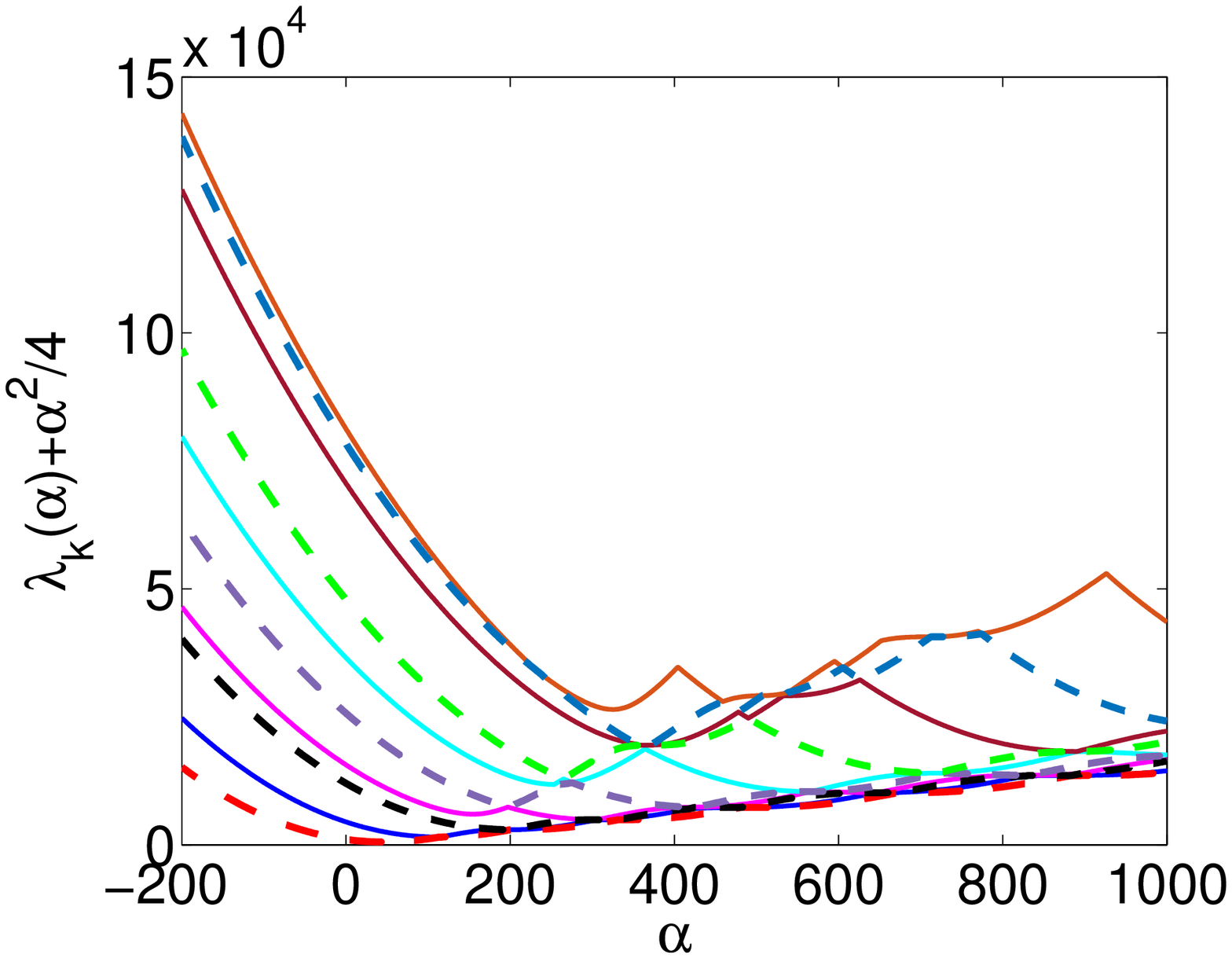}
\includegraphics[width=0.47\textwidth]{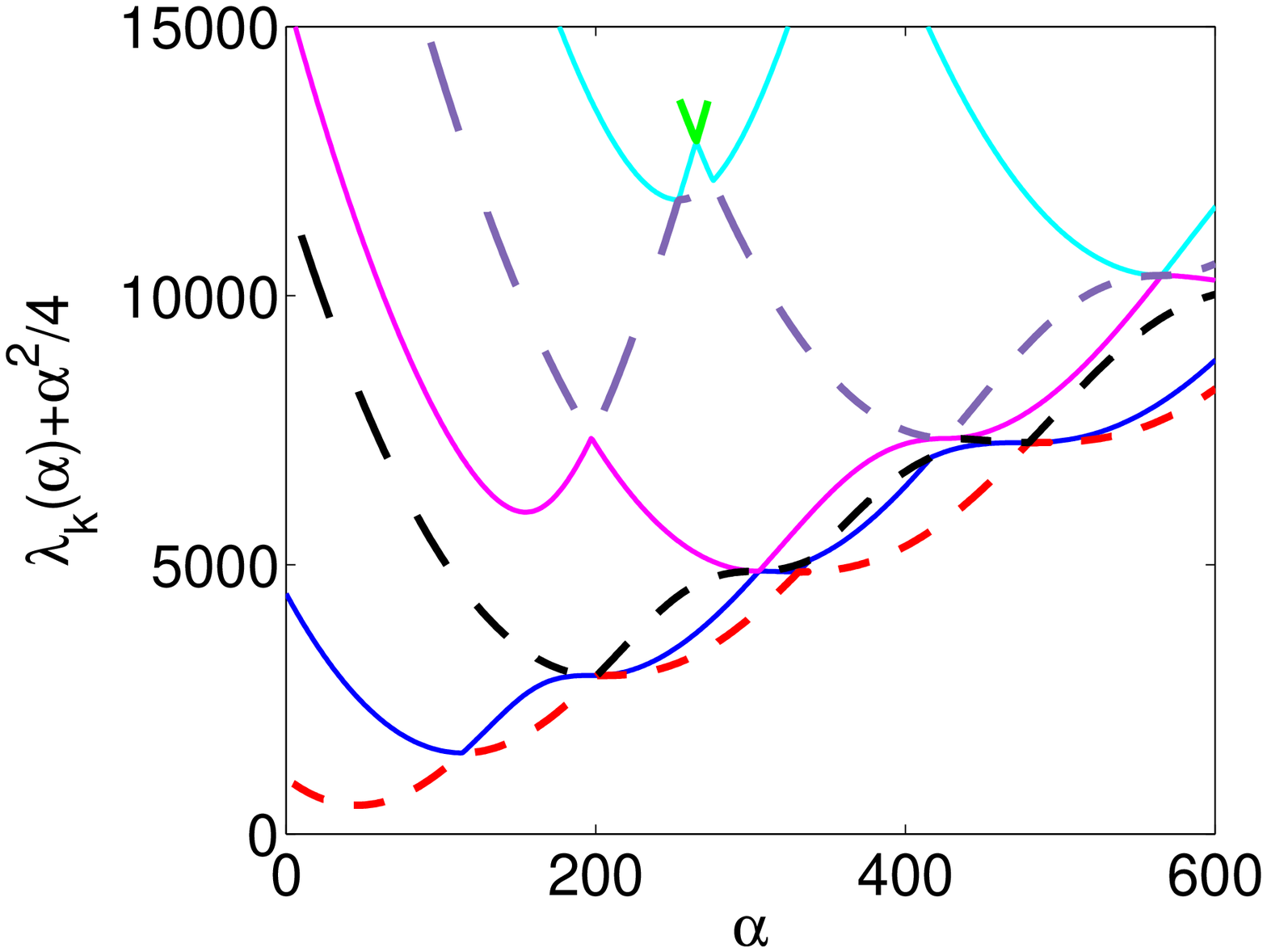}
(b) \includegraphics[width=0.47\textwidth]{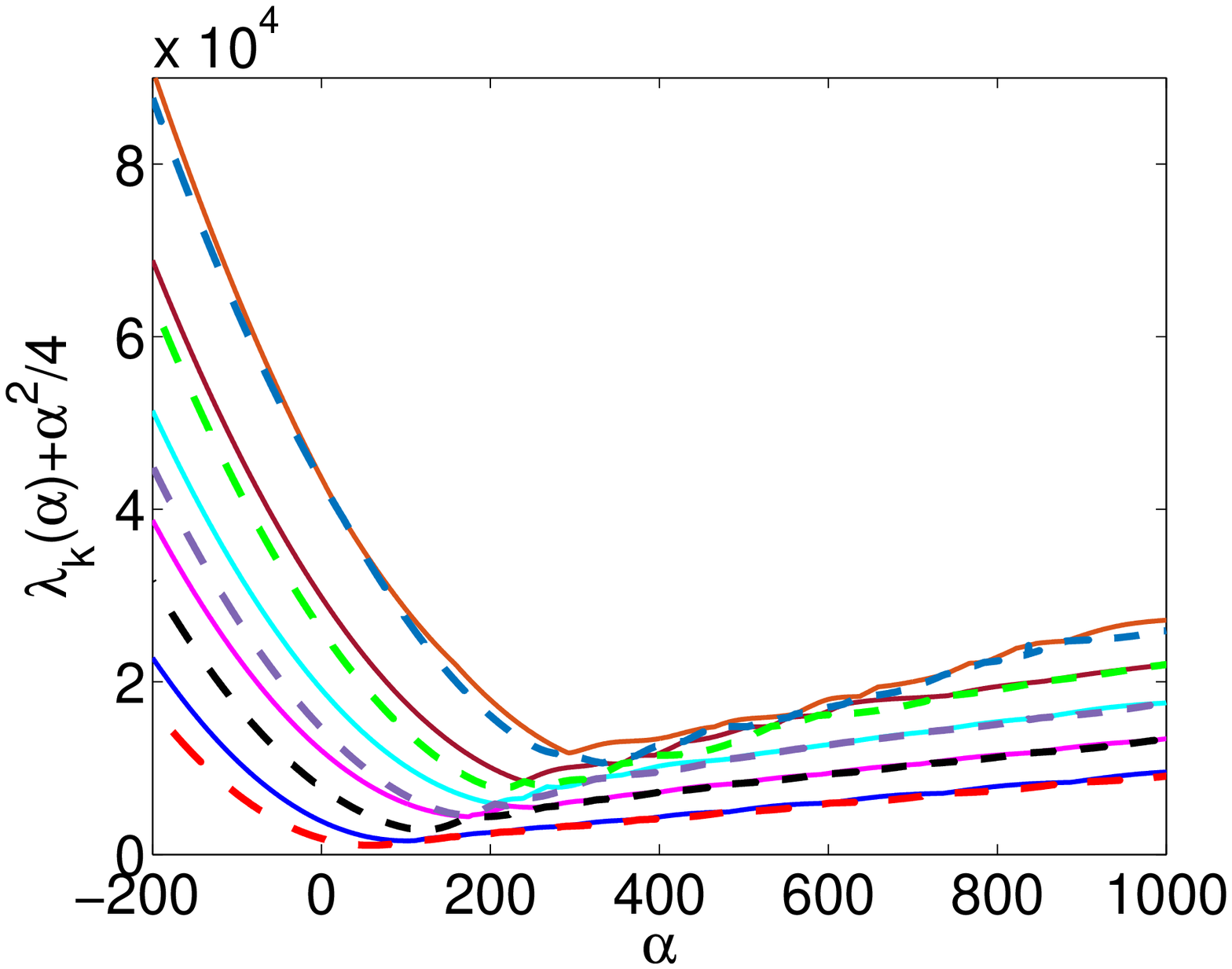}
\includegraphics[width=0.47\textwidth]{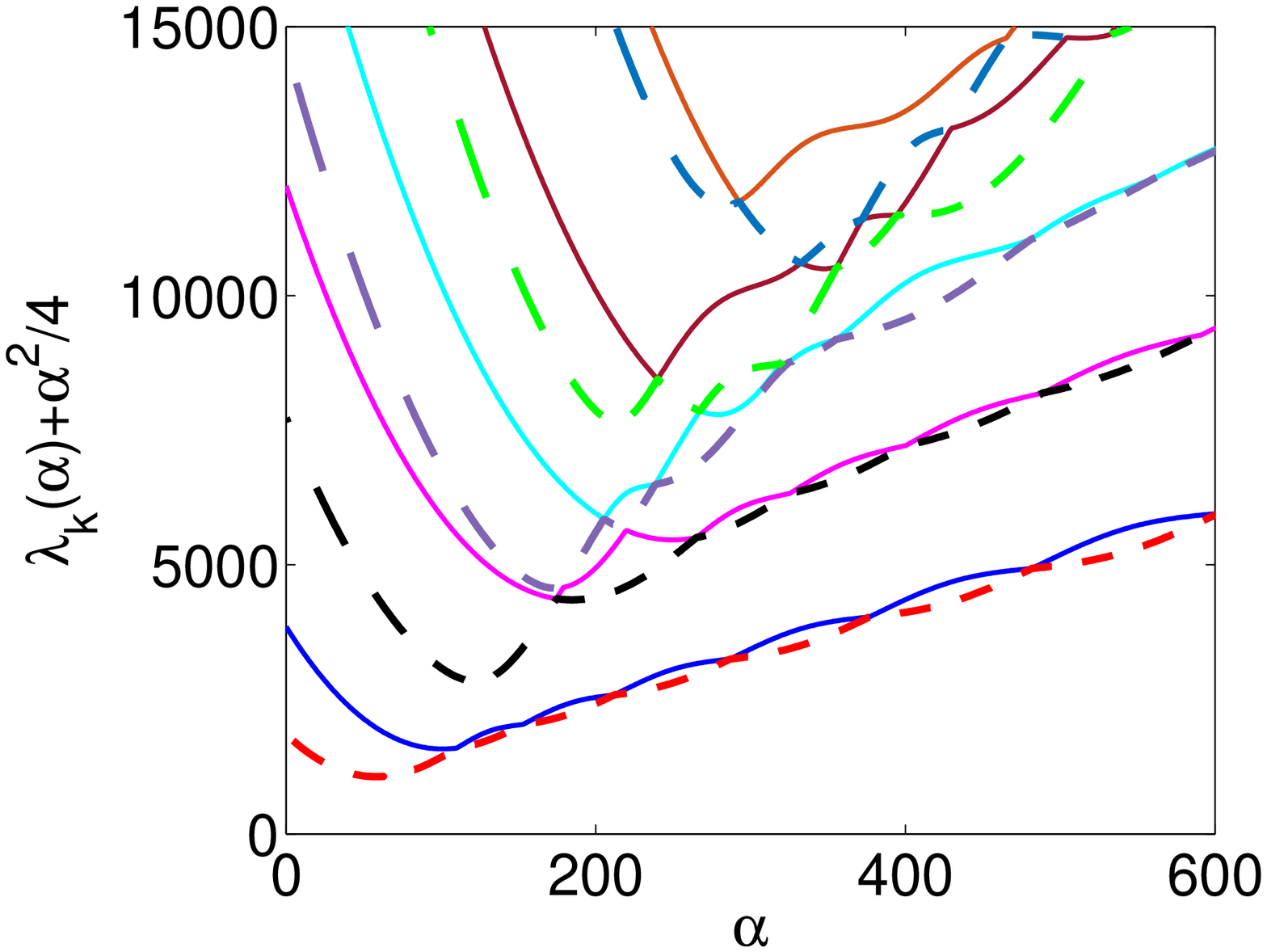}
\caption{(a) Plot of the quantities $\lambda_k(\Omega,\alpha)+\fr{\alpha^2}{4}, k=1,2,...,10$ for the disk with unit area, for $\alpha\in[-200,1000]$ (left plot)
and a zoom for $\alpha\in[0,600]$, illustrating the behaviour of the smallest eigenvalues, as a function of $\alpha$ (right plot). (b) Similar results for an ellipse
with unit area and eccentricity equal to $\sqrt{3}/2$.}
\label{fig:figurasde}
\end{figure}
This branch-switching phenomenon makes it much more difficult to obtain further terms in the asymptotic expansion and it is the independence of the
first term on the order of the eigenvalue which allows us to derive the expansion for all $k$. In the particular case where $\Omega$ is a ball of
radius $R$, which is at the heart of the proof of Theorem~\ref{mainthm}, we are able to prove that the number of such eigenvalue branches which
make up the $k^{\rm th}$ eigencurve is finite, and we determine further terms in the asymptotic expansion of these analytic branches. These results
are summarised in the following
\begin{theorem}[Asymptotic behaviour of analytic eigenvalue branches for balls]
\label{asymptball}
For any analytical branch of the eigenvalues of problem~\eqref{dirichlet} when $\Omega$ is a ball of radius $R$, we have
\begin{equation}
\lambda=-\fr{\alpha^2}{4}+\fr{c_{1} \alpha}{R^2}+\fr{c_{2}}{R^4}+\so(1),
\end{equation}
as $\alpha\to+\infty$, where $c_{1}$ and $c_{2}$ are constants depending on the eigenvalue branch, with $c_{1}$ being positive.
In the case of the first eigenvalue we have
\[
 \lambda_{1} = -\fr{\alpha^2}{4}+\fr{\pi^2 \alpha}{2R^2}+\fr{\pi^2(N^2-1-\pi^2)}{4R^4}+\so(1),
\]

\end{theorem}
\noindent The full description of the coefficients $c_{1}$ and $c_{2}$ may be found in Theorem~\ref{asympt} in Section~\ref{robinrelation}.

It is possible to consider problem~\eqref{dirichlet} with other boundary conditions, such as the Navier setting. This is not as interesting
from a mathematical perspective, since the problem then reduces directly to the study of the second order elliptic operator $\Delta + \alpha/2$.
However, and as we show in Section~\ref{sec:navier}, there is a major difference between the Dirichlet and Navier cases in that for the Navier
problem the number of crossings of analytic branches to make up an eigencurve corresponding to the $k^{\rm th}$ eigenvalue is actually infinite
for each $k$. Complex crossing and avoided-crossing patterns seem to be a characteristic of such systems in the large compression regime,
and they have also been identified in the one-dimensional fourth-order problem with different boundary conditions studied in~\cite{chaschun}.

Concerning our second topic of study, namely extremal domains for eigenvalues of problem~\eqref{dirichlet}, even in the case where
the parameter $\alpha$ vanishes the problem is known to be extremely difficult with results available only in two and three dimensions --
see~\cite{nadi,ashb}, respectively; for~\eqref{dirichlet2} there are no complete results in any dimension. Once $\alpha$ is taken to be
nonzero in~\eqref{dirichlet}, the only existing result is an extension to sufficiently small positive values of $\alpha$ in two
dimensions~\cite{ABM}. Our purpose in this part is thus mainly to provide a numerical exploration of the different types of extremal domains
under an area restriction, showing in particular that the ball is no longer a minimizer for large compression.

We consider the numerical solution of the eigenvalue problem~\eqref{dirichlet} using the Method of Fundamental Solutions
(see e.g., \cite{AA,Ant}). This is a meshless numerical method where the approximation is made by a discretization of an expansion in terms
of the single and double layer potentials. In particular, by construction, the numerical approximation satisfies the fourth-order partial
differential equation and we can focus just on the approximation of the boundary conditions of the problem. The computational implementation
of this numerical method is described in Section~\ref{nummeth} and some numerical results for the shape optimization problem are presented in 
Section~\ref{numres}. In particular, we will study minimizers of the first eigenvalue of problem~\eqref{dirichlet} subject to an 
area constraint. The obtained numerical results suggest that the minimizer depends on the parameter $\alpha$, with the ball being the minimizer for
all negative $\alpha$ and then extending to $\alpha\in[0,\alpha^\star]$, for some positive $\alpha^\star$. Note that this last result
corresponds to that proved in~\cite{ABM} for sufficiently small $\alpha$, with our numerical simulations suggesting that, in fact, one may
take $\alpha^{\star}$ to be at least as large as the first buckling eigenvalue.
For large values of the parameter $\alpha$ we obtain some non-trivial minimizers - see Figure~\ref{fig:figuras2}.

This numerical study has been performed mainly among general simply connected domains. However, we performed also the optimization of the first
eigenvalue of problem~\eqref{dirichlet} among annuli having unit area and compared the optimal values that were obtained with the corresponding
values of the ball. These results suggest that the first eigenvalue of the disk is always smaller than the corresponding eigenvalue of
the optimal annulus, independently of the parameter $\alpha$.

\section{Statement of the problem\label{sec2}}

We start by observing that problem~\eqref{dirichlet} has the following weak formulation
\begin{equation}
\label{weakdir}
\int_{\Omega}\Delta u\Delta\phi-\alpha\nabla u\nabla\phi=\lambda\int_\Omega u\phi,\ \forall \phi\in H^2_0(\Omega),
\end{equation}
and its eigenvalues may be described through their variational characterizations
\begin{equation}\label{rayleigh}
\lambda_k^D(\Omega,\alpha)=\min_{\substack{V\subset H_0^2(\Omega)\\ \dim V=k}} \max_{0\neq u\in V}
\frac{\int_\Omega (\Delta u)^2-\alpha |\nabla u|^2}{\int_\Omega u^2}.
\end{equation}
In what follows, whenever the meaning is clear from the context, we will drop either argument in $\lambda_k^D(\Omega,\alpha)$ for the
sake of simplicity.

In order to determine the eigenfunctions of problem~\eqref{dirichlet} when $\Omega=B_R(0)$, we rewrite equation~\eqref{dirichlet} as
\begin{equation}
\label{eq}
(\Delta+\ap)(\Delta+\am )u=0,
\end{equation}
where
\begin{equation}
\label{alpha}
\ap=\frac\alpha 2+\sqrt{\frac{\alpha^2}{4}+\lambda},\ \ \am =\frac\alpha 2-\sqrt{\frac{\alpha^2}{4}+\lambda}.
\end{equation}
Both $\ap$ and $\am$ are always real, as may be seen from inequality \eqref{dirnav}, and
$\ap$ is always positive while the sign of $\am$ depends on the sign of the eigenvalue $\lambda$.

For positive $\lambda$ it is known that the solution of~\eqref{eq} can be written as (cf.\ \cite{ashb})
\begin{equation}
\label{positive_lambda}
u(r,\theta)=r^{1-\frac N 2}\left[A J_{k+\frac N 2 -1}\left(r\sqrt{\ap}\right)+B I_{{k}+\frac N 2 -1}\left(r\sqrt{-\am }\right)\right]S_k(\theta),
\end{equation}
where $J_\nu$ and $I_\nu$ are the Bessel and the modified Bessel functions, respectively, of the first kind of order $\nu$, and $S_k$ are the spherical
harmonic functions of order $k$. The boundary conditions then yield the following system of equations
\begin{equation}
\label{sys}
\left\{\begin{array}{l}
A f_k(R) +B g_k(R)=0,\\
A f_k'(R) +B g_k'(R)=0,
\end{array}\right.
\end{equation}
where we have set
$$f_k(r)=r^{1-\frac N 2}J_{k+\frac N 2 -1}\left(r\sqrt{\ap}\right) \mbox{ and } g_k(r)=r^{1-\frac N 2}I_{k+\frac N 2 -1}\left(r\sqrt{-\am }\right).$$
Since we are interested in the existence of nontrivial solutions of system~\eqref{sys}, we impose the corresponding determinant to be zero, namely
\begin{equation}
\label{bessel1}
{f_k(R)}{g_k'(R)}-{g_k(R)}{f_k'(R)}=0,
\end{equation}
from which we obtain the corresponding eigenvalues and, as a consequence, the general form of the eigenfunctions. Using standard Bessel function identities,
equation~\eqref{bessel1} may be rewritten as
\begin{equation} \label{bessel2}
\begin{array}{l}
\fr{RJ_{k+\frac{N}{2} -1}\left(R\sqrt{\ap}\right)}{kJ_{k+\frac N 2-1}\left(R\sqrt{\ap}\right)-R\sqrt{\ap}J_{k+\frac N 2}\left(R\sqrt{\ap}\right)} \eqskip
\hspace*{3cm}= \fr{RI_{k+\frac{N}{2} -1}\left(R\sqrt{-\am }\right)}{kI_{k+\frac N 2-1}\left(R\sqrt{-\am }\right)+R\sqrt{-\am }I_{k+\frac{N}{2}}\left(R\sqrt{-\am }\right)}.
\end{array}
\end{equation}

When $\lambda$ is strictly negative, then $\am $ is strictly positive and, in place of~\eqref{positive_lambda}, we now have
\begin{equation}
\label{negative_lambda}
u(r,\theta)=r^{1-\frac N 2}\left[A J_{k+\frac N 2 -1}\left(r\sqrt{\ap}\right)+B J_{k+\frac N 2 -1}\left(r\sqrt{\am }\right)\right]S_k(\theta),
\end{equation}
where the coefficients are given by a system similar to~\eqref{sys}, and the eigenvalues are now solutions of
\begin{equation}\label{bessel3}
\begin{array}{ll}
\fr{J_{k+\frac N 2 -1}\left(R\sqrt{\ap}\right)}{kRJ_{k+\frac N 2-1}\left(R\sqrt{\ap}\right)-\sqrt{\ap}J_{k+\frac N 2}\left(R\sqrt{\ap}\right)}\eqskip
\hspace*{3cm} = \fr{J_{k+\frac N 2 -1}\left(R\sqrt{\am }\right)}{kRJ_{k+\frac N 2-1}\left(R\sqrt{\am }\right)-\sqrt{\am }J_{k+\frac N 2}
\left(R\sqrt{\am }\right)}.
\end{array}
\end{equation}

Finally, it remains to consider the case $\lambda=0$, which behaves in a slightly different way. We note that then $\ap=\alpha$ while $\am =0$, and
in particular this means that $\alpha$ has to be an eigenvalue of the following buckling problem
\begin{equation}
\label{buckling}
\left\{\begin{array}{ll}
\Delta^2u=-\Lambda\Delta u, & {\rm \ in\ }\Omega,\\
u=\frac{\partial u}{\partial \nu}=0, & {\rm \ on\ }\partial \Omega,
\end{array}\right.
\end{equation}
for which the eigenfunctions are known to be of the form (they can be derived in a similar way as for the other cases)
\begin{equation}
\label{zero_lambda}
u(r,\theta)=\left[A r^{1-\frac N 2} J_{k+\frac N 2 -1}\left(r\sqrt{\alpha}\right)+B r^k\right]S_k(\theta).
\end{equation}
In particular, $\alpha$ has to be a solution of the following
\begin{equation}
\label{bessel_buckling}
J_{k+\frac N 2}\left(R\sqrt{\alpha}\right)=0.
\end{equation}
Moreover, if $\alpha$ is the $k$-th eigenvalue $\Lambda_k$ of the buckling problem~\eqref{buckling}, we immediately deduce that the vanishing eigenvalue of problem~\eqref{dirichlet} is exactly the $k$-th one $\lambda_k$, and the multiplicity will be the same of $\Lambda_k$.

\section{Connection to the Robin Laplacian\label{robinrelation}}

Even though there are no simple relations between the Laplacian and the Bilaplacian in general (apart for the Navier problem~\eqref{navier}), if we consider the
generic situation of problem~\eqref{dirichlet} in the ball in $\mathbb R^N$ with $\alpha\in \mathbb R$, we can draw a very precise connection to the Robin Laplacian.

To this end we recall that the Robin problem for the Laplace operator is as follows
\begin{equation}
\label{robin}
\left\{\begin{array}{ll}
-\Delta u=\sigma u, & {\rm \ in\ }\Omega,\\
\frac{\partial u}{\partial \nu}+\beta u=0, & {\rm \ on\ }\partial \Omega.
\end{array}\right.
\end{equation}
For any real value of $\beta$ the corresponding spectrum consists of a non-decreasing sequence of eigenvalues with finite multiplicities diverging to plus infinity.
In particular, for positive values of $\beta$ the eigenvalues are all strictly positive, while for $\beta=0$
the Robin problem~\eqref{robin} becomes the
Neumann problem. It is also known that, as $\beta\to+\infty$, problem~\eqref{robin} converges to the Dirichlet problem for the Laplace operator, namely
\begin{equation}
\label{dirlap}
\left\{\begin{array}{ll}
-\Delta u=\gamma u, & {\rm \ in\ }\Omega,\\
u=0, & {\rm \ on\ }\partial \Omega,
\end{array}\right.
\end{equation}
whose eigenvalues we will denote by
\[
 0<\gamma_{1}\leq \gamma_{2}\leq \dots \to +\infty.
\]
We recall that the eigenfunctions of~\eqref{robin} on a ball can be sorted out into three categories (cf.\ Section \ref{sec2}):
\begin{itemize}
\item[i.] if the eigenvalue $\sigma$ is positive, the eigenfunction is of the form
$$
r^{1-\frac N 2}J_{k+\frac N 2 -1}(r\sqrt{\sigma})S_k(\theta),
$$
and eigenvalues are solutions of
$$
\left(\frac k R +\beta\right)J_{k+\frac N 2-1}(R\sqrt{\sigma})=\sqrt{\sigma}J_{k+\frac N 2}(R\sqrt{\sigma});
$$
\item[ii.] if the eigenvalue $\sigma$ is negative, the eigenfunction is of the form
$$
r^{1-\frac N 2}I_{k+\frac N 2 -1}(r\sqrt{-\sigma})S_k(\theta),
$$
and eigenvalues are solutions of
$$
\left(\frac k R +\beta\right)I_{k+\frac N 2-1}(R\sqrt{-\sigma})=-\sqrt{-\sigma}I_{k+\frac N 2}(R\sqrt{-\sigma});
$$
\item[iii.] if the eigenvalue $\sigma$ is zero, the eigenfunction is of the form $r^kS_k(\theta)$, and in particular this occurs when $\beta=-\frac k R$.
\end{itemize}

At this point it is clear that any eigenfunction of the clamped plate problem~\eqref{dirichlet} on the ball can be thought of as the sum of two different
eigenfunctions of the Robin problem for the Laplacian~\eqref{robin}. A first condition that these two Robin eigenfunctions have to satisfy is that their
spherical parts coincide. This implies that they must come from two different eigenvalues and, in particular, we have that these two eigenvalues are $\ap$ and $\am $,
and the multiplicities must coincide. Furthermore, such eigenfunctions must have the same index $\nu=k+\frac N 2 -1$ in their Bessel function part. Let us call $v_1,v_2$
two such eigenfunctions (with associated eigenvalues $\sigma_1,\sigma_2$) and let
$$
v_j(r,\theta)=v_j^R(r)S(\theta),
$$ 
i.e., we denote by $v_j^R$ the radial part. Since we want the boundary conditions in the clamped plate problem~\eqref{dirichlet} to be satisfied, the only way to
combine $v_1$ and $v_2$ is to set
\begin{equation}
\label{robinbilaplacian}
u=v_1^R(R)v_2-v_2^R(R)v_1,
\end{equation}
as can be easily checked from the boundary conditions in the Robin problem for the Laplace operator~\eqref{robin}. In particular, $v_1$ and $v_2$ must be Robin eigenfunctions associated with
the same parameter $\beta$. As for the equation, we observe that
$$
\Delta^2 v_j +\alpha\Delta v_j=(\sigma_j^2-\alpha\sigma_j)v_j
$$
and the equality $\sigma_1^2-\alpha\sigma_1=\sigma_2^2-\alpha\sigma_2$ is naturally satisfied since
\begin{equation}
\label{robintodir}
\alpha=\ap+\am =\sigma_1+\sigma_2,\ \ \ \lambda=-\ap\am =-\sigma_1\sigma_2.
\end{equation}

On the other hand, letting $\beta$ go to infinity we get that, for specific values of $\alpha$ and $\lambda$, we should consider eigenvalues of the
Dirichlet Laplacian~\eqref{dirlap} instead. From the literature (see e.g.,~\cite{bfk} and the references therein, and also~\cite{fl} for a study of the
first two Robin eigenvalues) we know that all the analytical branches related to the same Bessel function
$J_{k+\frac N 2-1}$ ($I_{k+\frac N 2-1}$ when the eigenvalue is negative, $r^k$ if zero) can be continued at $\beta=\infty$ generating a function
which wraps around $\mathbb R$ infinitely many times. If we call $\sigma_{k,j}(\beta)$ the $j$-th eigenvalue associated with $J_{k+\frac N 2-1}$,
then the analytical branches of eigenvalues of problem~\eqref{dirichlet} are given by
\begin{equation}
\label{eigen}
-\sigma_{k,j}(\beta)\sigma_{k,j+t}(\beta),
\end{equation}
for some $t\in\mathbb N$, where the parameter $j$ is of no relevance here since any time $\beta$ reaches infinity $j$ has to be
replaced by $j+1$ as
$$
\gamma_{k,j}=\sigma_{k,j}(+\infty)=\sigma_{k,j+1}(-\infty),
$$
where $\gamma_{k,j}$ is the $j$-th eigenvalue of the Dirichlet problem for the Laplace operator~\eqref{dirlap} associated with $J_{k+\frac N 2-1}$.
In particular, different branches of
eigenvalues of the clamped plate problem~\eqref{dirichlet} associated with the Bessel index $k+\frac N 2-1$ are indexed by the parameter $t$ in~\eqref{eigen}.
We remark that all the branches are of this type, hence no other branches are present. We sum up all these arguments in the following

\begin{theorem}
\label{robintobilaplace}
Let $\beta\in\mathbb R$ and $v_1$ and $v_2$ be any two eigenfunctions of problem~\eqref{robin} in a ball $B_R(0)$ associated with the eigenvalues
$\sigma_1$ and $\sigma_2$, respectively, and having the same spherical part, namely
$$
v_j(x)=v_j^R(r)S(\theta),\quad j=1,2.
$$
Then the function defined in~\eqref{robinbilaplacian} is an eigenfunction of problem~\eqref{dirichlet} in the ball $B_R(0)$ associated with the 
eigenvalue $\lambda=-\sigma_1\sigma_2$ and with the parameter $\alpha=\sigma_1+\sigma_2$.

This representation completely characterizes the analytic branch of the eigenvalue $\lambda=-\sigma_1\sigma_2$ (for $\alpha\in\mathbb R$) as the parameter $\beta$ varies.
In the limits $\beta\to\pm\infty$ we have that the eigenvalue $\lambda$ can be written as a product of eigenvalues of 
the Dirichlet problem for the Laplace operator~\eqref{dirlap}, and the corresponding eigenfunction can also be written as a combination of eigenfunctions of
problem~\eqref{dirlap}.

In addition, all analytic branches of eigenvalues of the clamped plate problem~\eqref{dirichlet} can be represented in this fashion.
\end{theorem}

Theorem \ref{robintobilaplace} allows us to study the behaviour of the eigenvalues as $\alpha=\pm\infty$. Actually, for the case $\alpha=-\infty$, the
convergence is well known in the literature for any smooth domain (see e.g., \cite[p. 392]{frank}).
\begin{theorem}
\label{frankthm}
Let $w_\alpha$ be the eigenfunction associated with $\lambda_k(\alpha)$, and suppose that there exists a point $\alpha_0\in\mathbb R$ such that $w_\alpha\in C^5(\Omega)$
for any $\alpha<\alpha_0$. Then
\begin{equation}
\label{frankasymptotics}
\lambda_k(\alpha)=-\alpha\gamma_k+\sqrt{-\alpha}\int_{\partial\Omega}\left|\nabla u_k\right|^2+\bo(1),
\end{equation}
as $\alpha\to-\infty$, where $u_k$ is an eigenfunction of the Dirichlet problem for the Laplace operator~\eqref{dirlap} associated with $\gamma_k$.
\end{theorem}

We recall that, thanks to classical regularity theory for elliptic operators (cf.\cite{ggs}), if $\Omega\in C^{5,\delta}$ then
$w_\alpha\in C^{5,\delta}(\Omega)$ for any $\alpha\in \mathbb R$ and, in particular, balls satisfy the hypotheses of Theorem \ref{frankthm}.
It is easily seen then that we can recover the first term of the asymptotics~\eqref{frankasymptotics} using the known asymptotics for the
Robin problem (see \cite{pankpop} and the references therein).

Regarding the asymptotics as $\alpha\to+\infty$, we compute it using the knowledge that for any given branch when we get to $\beta=\infty$ we obtain that both $\alpha$
and $\lambda$ can be expressed in terms of zeros of Bessel functions:
\begin{equation}
\alpha=\frac{j_{k+\frac N 2-1,m}^2+j_{k+\frac N 2-1,m+t}^2}{R^2},\ \ \ \lambda=-\frac{j_{k+\frac N 2-1,m}^2\times j_{k+\frac N 2-1,m+t}^2}{R^4},
\label{besseldir}
\end{equation}
where $j_{\nu,m}$ is the $m$-th zero of $J_{\nu}$, whose asymptotic behaviour is known to be (cf.\ \cite[formula (10.21.19)]{nist})
\begin{equation}
\label{nistasympt}
j_{\nu,m}\sim\left(m+\frac \nu 2 -\frac 1 4\right)\pi-\frac{4\nu^2-1}{8\left(m+\frac \nu 2 -\frac 1 4\right)\pi} + \so\left(\frac 1 {m^2}\right)
\end{equation}
as $m\to\infty$.

Let us now denote by $\psi_{m}$ and $\psi_{m+t}$ two eigenfunctions of the Dirichlet problem for the Laplace operator~\eqref{dirlap} associated
with $\gamma_m=R^{-2}j_{k+\frac N 2-1,m}^2$
and $\gamma_{m+t}=R^{-2}j_{k+\frac N 2-1,m+t}^2$, respectively, having the same spherical part, and normalized such that $\psi_m+\psi_{m+t}$ is an
eigenfunction of the clamped plate problem~\eqref{dirichlet} under condition~\eqref{besseldir}. Then using the Rayleigh quotient representation of
$\lambda$ we have
\begin{equation}
\label{remainder1}
\begin{array}{lll}
\lambda & = & -\fr{\alpha^2}{4}+\fr{\dint_{B_R}\left[\Delta(\psi_m+\psi_{m+t})+\fr{\alpha}{2}(\psi_m+\psi_{m+t})\right]^2}{\dint_{B_R}(\psi_m+\psi_{m+t})^2}\eqskip
& = & -\fr{\alpha^2}{4}+\left(\fr{\gamma_m-\gamma_{m+t}}{2}\right)^2.
\end{array}
\end{equation}
We recall that in this particular case we have $\alpha=(j_{\nu,m}^2+j_{\nu,m+t}^2)R^{-2}$, where we set $\nu=k+\frac N 2-1$ for simplicity.
We now compute the asymptotics for the remainder in~\eqref{remainder1} and get
\[
\begin{array}{lll}
\fr{\left(j_{\nu,m}^2-j_{\nu,m+t}^2\right)^2}{4R^2\left(j_{\nu,m}^2+j_{\nu,m+t}^2\right)} &
\approx & \fr{\left[\left(m+\frac \nu 2 -\fr{1}{4}\right)^2\pi^2-\left(m+t+\fr{\nu}{2} -\fr{1}{4}\right)^2\pi^2\right]^2}{4R^2\left[\left(m+\frac \nu 2 -\fr{1}{4} \right)^2\pi^2+\left(m+t+\frac \nu 2 -\fr{1}{4}\right)^2\pi^2\right]}\eqskip
& \approx & \fr{t^2\pi^2}{2R^2},
\end{array}
\]
telling us that $\lambda\sim -\fr{\alpha^2}{4}+\fr{\alpha t^2\pi^2}{2R^2}+\so(\alpha)$. Going further we can get
\[
\begin{array}{l}
\fr{\left(j_{\nu,m}^2-j_{\nu,m+t}^2\right)^2}{4R^4}-\fr{\left(j_{\nu,m}^2+j_{\nu,m+t}^2\right)t^2\pi^2}{2R^4}\eqskip
\hspace*{2cm}\approx\fr{\pi^4}{4R^4}\left[\left(m+\frac \nu 2 -\frac 1 4\right)^2\pi^2-\left(m+t+\frac \nu 2 -\frac 1 4\right)^2\pi^2\right]^2\eqskip
\hspace*{3cm}-\fr{t^2\pi^4}{2R^4}\left[\left(m+\fr{\nu}{2} -\fr{1}{4}\right)^2+\left(m+t+\fr{\nu}{2} -\fr{1}{4}\right)^2-\fr{4\nu^2-1}{2\pi^2}\right]\eqskip
\hspace*{2cm}\approx \fr{t^2\pi^2(4\nu^2-1-t^2\pi^2)}{4R^4},
\end{array}
\]
and hence we have
\begin{theorem}
\label{asympt}
For any analytical branch of the eigenvalues of problem~\eqref{dirichlet} on a ball $B_R$ of radius $R$, we have
\begin{equation}
\label{asymptotic}
\lambda=-\frac{\alpha^2}{4}+\frac{\alpha t^2\pi^2}{2R^2}+\frac{t^2\pi^2(4\nu^2-1-t^2\pi^2)}{4R^4}+\so(1),
\end{equation}
as $\alpha\to+\infty$, where $\nu=k+\frac N 2 -1$ is the index of the associated Bessel functions, and $t$ is the parameter introduced in~\eqref{eigen}.
\end{theorem}

We observe that, even if at a first glance the presence of the parameter $t$ may seem unnatural, it may be compared for example with the ordering
number for zeros of Bessel functions $j_{\nu,k}$. From this perspective, it is natural that it appears in formula~\eqref{asymptotic}.

We are now ready to prove Theorem~\ref{mainthm}.

\begin{proof}[Proof of Theorem~\ref{mainthm}]
In the case of a general domain $\Omega$, we shall denote the radius of the largest inscribed ball and that of the smallest ball containing $\Omega$ by $R_i,R_c$ respectively.
By the inclusion properties for problem~\eqref{dirichlet}, we know that any eigenvalue of $\Omega$ is bounded from above and from below by the 
corresponding eigenvalues of the inscribed and circumscribed balls, respectively. This immediately proves~\eqref{gendom}. For higher eigenvalues
it will, in general, be difficult to determine the precise order of each eigenvalue, but in the case of the first eigenvalue it is possible to
identify the corresponding branch, namely that obtained by making $t=1$ and $k=0$, and in turn obtain the following (asymptotic) expression
$$
-\frac{\alpha^2}{4}+\frac{\alpha \pi^2}{2R_c^2}\lesssim\lambda_1(\Omega)\lesssim-\frac{\alpha^2}{4}+\frac{\alpha\pi^2}{2R_i^2},
$$
which implies~\eqref{gendom1}.
\end{proof}

\section{The Navier problem\label{sec:navier}}

We now turn our attention to the following eigenvalue problem
\begin{equation}
\label{navier}
\left\{\begin{array}{ll}
\Delta^2u+\alpha\Delta u=\lambda u, & {\rm \ in\ }\Omega,\\
u=\Delta u=0, & {\rm \ on\ }\partial \Omega,
\end{array}\right.
\end{equation}
for any $\alpha\in\mathbb R$. We immediately notice the resemblance of the Navier problem~\eqref{navier} with problem~\eqref{dirichlet}, as its weak formulation reads
\begin{equation}
\int_{\Omega}\Delta u\Delta\phi-\alpha\nabla u\nabla\phi=\lambda\int_\Omega u\phi,\ \forall \phi\in H^2(\Omega)\cap H^1_0(\Omega),
\end{equation}
the only difference between this and~\eqref{weakdir} being the ambient space. In particular, comparing the variational
characterization~\eqref{rayleigh} of the eigenvalues of problem~\eqref{dirichlet} with that of the eigenvalues of the Navier problem~\eqref{navier}
\begin{equation}\label{rayleighnav}
\lambda_k^N(\Omega,\alpha)=\min_{\substack{V\subset H^2(\Omega)\cap H^1_0(\Omega)\\ \dim V=k}} \max_{0\neq u\in V}
\frac{\int_\Omega (\Delta u)^2-\alpha |\nabla u|^2}{\int_\Omega u^2},
\end{equation}
yields
$$
\lambda_k^D(\Omega,\alpha)\ge \lambda_k^N(\Omega,\alpha),\ \ \forall k\in\mathbb N, \ \forall \alpha\in\mathbb R.
$$

Now we want to compute eigenfunctions and eigenvalues of the Navier problem~\eqref{navier}. We can of course proceed as for the Dirichlet case in Section \ref{sec2}.
However, we observe that we can modify the problem as follows
\begin{equation}
\label{naviermod}
\left\{\begin{array}{ll}
\Delta^2u+\alpha\Delta u+\fr{\alpha^2}4 u=\left(\lambda+\fr{\alpha^2}{4}\right) u, & {\rm \ in\ }\Omega,\\
u=\Delta u +\fr{\alpha}2 u=0, & {\rm \ on\ }\partial \Omega,
\end{array}\right.
\end{equation}
which tells us immediately that, if the domain has the cone property, the Navier operator in~\eqref{naviermod} is the square of the translated Dirichlet Laplace operator
$\Delta+\frac\alpha 2$ (cf.\ \cite{ggs}). In particular, if we denote by $\gamma_k$ the $k$-th eigenvalue of the Dirichlet Laplacian~\eqref{dirlap}, we get
that the spectrum of~\eqref{navier} is given by
\begin{equation}
\left\{\gamma_k^2(\Omega)-\alpha \gamma_k(\Omega)\right\}_{k}
\end{equation}
for any $\alpha\in\mathbb R$ and for any (smooth enough) domain $\Omega$. We remark that, for $\alpha<0$ (actually, for $\alpha<2\gamma_1$) we have
$$
\lambda_k^N(\alpha)=\gamma_k^2-\alpha\gamma_k
$$
for any $k$, while on the other hand we actually have intersections of the branches (the intersection points will depend on $\Omega$).
However, we can still say that
$$
\lambda_1^N(\alpha)=\min_k\{\gamma_k^2-\alpha\gamma_k\}=\min_k\{(\gamma_k^2-\fr{\alpha}{2})^2\}-\fr{\alpha^2}{4}\ge-\fr{\alpha^2}{4},
$$
that is
\begin{equation}
\label{dirnav}
\lambda_k^D(\alpha)\ge\lambda_k^N(\alpha)\ge\lambda_1^N(\alpha)\ge-\frac{\alpha^2}{4},\ \ \forall \alpha\in\mathbb{R}.
\end{equation}

\begin{theorem}
Let $\Omega$ be a bounded open set in $\mathbb R^N$ with the cone property. Then, for any $k\in\mathbb N$
\begin{equation}
\label{asymnav}
\lambda_k^N(\alpha)=-\frac{\alpha^2}4+\so(\alpha^2),
\end{equation}
as $\alpha\to+\infty$. Moreover,
\begin{equation}
\label{asymnav1}
\lambda_1^N(\alpha)=-\frac{\alpha^2}4+\so(\alpha),
\end{equation}
as $\alpha\to+\infty$.
\end{theorem}

\begin{proof}
Equality~\eqref{asymnav} easily follows from the inequality chain~\eqref{dirnav} coupled with the asymptotic expansion~\eqref{gendom}.

As for~\eqref{asymnav1}, we first observe that
$$
\lambda_1^N(\alpha)=\gamma_k^2-\alpha\gamma_k\ \ \text{for}\ \gamma_{k-1}+\gamma_k\le\alpha\le\gamma_{k}+\gamma_{k+1},
$$
and for the choice $\alpha=2\gamma_k$ we have
\begin{equation}
\label{tangency}
\lambda_1^N(\alpha)=-\gamma_k^2=-\frac{\alpha^2}{4}.
\end{equation}
This alone is not enough to prove the asymptotic behaviour. However, we know that $\lambda_1^N(\alpha)$ is a polygonal line and that each and every segment is tangent to the asymptotic curve (thanks to~\eqref{tangency}). It is thus enough to show that the vertices have the same asymptotic
behaviour, i.e., the points $\alpha=\gamma_k+\gamma_{k+1}$ for which
$$
\lambda_1^N(\alpha)=-\gamma_k\gamma_{k+1},
$$
or equivalently
$$
\lambda_1^N(\alpha)-\frac{\alpha^2}4=\frac{(\gamma_{k+1}-\gamma_k)^2}4,
$$
therefore we have to show that
\begin{equation}
\label{quotient}
\frac{(\gamma_{k+1}-\gamma_k)^2}{\gamma_{k+1}+\gamma_k}\to 0\ \ \text{as}\ k\to\infty.
\end{equation}
To this end we recall the Weyl asymptotics for the Dirichlet eigenvalue problem for the Laplacian~\eqref{dirlap}, namely,
\begin{equation}
\label{weyl}
\gamma_k=C_1 k^{\frac 2 N}+C_2 k^{\frac 1 N}+\so(k^{\frac 1 N})\ \ \text{as}\ k\to\infty,
\end{equation}
where $C_1$ and $C_2$ are (known) constants depending only on $\Omega$ and the dimension $N$. From the binomial Taylor expansion
$$
(k+1)^\delta=k^\delta+\delta k^{\delta-1}+\so(k^{\delta-1})\ \ \text{as}\ k\to\infty,
$$
we have
\begin{equation}
\label{limit}
\frac{(\gamma_{k+1}-\gamma_k)^2}{\gamma_{k+1}+\gamma_k}=\frac{(\frac{2C_1}Nk^{\frac 2 N-1}+\frac{C_2}Nk^{\frac 1 N-1}+\so(k^{\frac 1 N}))^2}{2C_1k^{\frac 2 N}+\so(k^{\frac 2 N})},
\end{equation}
which clearly goes to zero for $N$ larger than one.
\end{proof}

\begin{remark}
If the domain is not bounded, it is still possible to prove~\eqref{asymnav1} without using the asymptotics~\eqref{gendom} while following the same strategy we used in the previous proof. In
particular, in order to get the term $-\frac{\alpha^2}4$, it suffices to show that
$$
\frac{(\gamma_k+\gamma_{k+1})^2}{\gamma_k\gamma_{k+1}}\to 4\ \ \text{as}\ k\to\infty,
$$
which follows from the equality
$$
\frac{(\gamma_k+\gamma_{k+1})^2}{\gamma_k\gamma_{k+1}}=\frac{\gamma_k}{\gamma_{k+1}}+\frac{\gamma_{k+1}}{\gamma_k}+2
$$
and the fact that the ratio of consecutive eigenvalues converges to $1$, thanks to Weyl's asymptotics~\eqref{weyl}.

Also, it is clear from~\eqref{limit} that the term $\so(\alpha)$ in~\eqref{asymnav1} is sharp, since a different exponent in the denominator in the
limit~\eqref{quotient} would not go to zero as $k\to\infty$.
\end{remark}

\begin{remark}
We observe that the behaviour of the clamped plate problem~\eqref{dirichlet} and that of the Navier problem~\eqref{navier} are substantially different.
On the one hand, from the 
asymptotics~\eqref{asymptotic} we have that the branches of eigenvalues of the clamped plate problem~\eqref{dirichlet} will stop intersecting for some sufficiently
large value of $\alpha$, at least in the case of balls where the parameters $t$ and $k$ provide a clear ordering of the branches, so that it is in 
principle possible to see which branch will eventually be the $k$-th eigenvalue. On the other hand, we know a priori that the branches of
eigenvalues of the Navier problem~\eqref{navier} will have an infinite number of intersections, making it quite complicated to decide which
is the $k$-th eigenvalue. In particular, the knowledge of the behaviour of each individual branch does not provide sufficient information on the
asymptotics of the eigenvalues. Similarly, even though the eigenspaces do not depend on $\alpha$, that associated with the $k^{\rm th}$ eigenvalue
will keep on changing, creating a strange phenomenon of non-convergence.
\end{remark}


\section{Shape derivatives}

We will now consider the problem of finding extremal domains for the $k$-th eigenvalue of problem~\eqref{dirichlet}, namely,
\begin{problem}
\label{shapeoptdir}
Determine
\[\lambda^{\ast}_k(\alpha)=\inf_{\Omega\subset\mathbb{R}^n}\left\{\lambda_k(\Omega,\alpha):\left|\Omega\right|=1\right\}.\]
\end{problem}

We observe that proving existence for Problem~\ref{shapeoptdir} within a specific class of domains can be quite difficult and, to
the best of our knowledge, there are no results available in general. To gauge the difficulties involved, we refer the reader
to~\cite{bucur} for a survey on existence results for the Laplacian case, for which it is still not known if existence holds
within the class of open sets.

We will focus now on Problem~\ref{shapeoptdir} with $k=1$. We begin by deriving the formula for the Hadarmard shape derivative of an eigenvalue of~\eqref{dirichlet}. Note that the formula in the case $\alpha=0$ was already derived
in a general setting and for multiple eigenvalues, see \cite{bl, ortzua}. We also refer to \cite{buososurvey} and the references therein for a complete discussion on Hadamard
formulas for the Biharmonic operator, also in the case $\alpha\neq 0$. Nevertheless, for the sake of simplicity we show here how to derive it in our specific case.

Consider an application $\Psi(t)$ such that $\Psi:t\in[0,T[\rightarrow W^{1,\infty}(\mathbb{R}^N,\mathbb{R}^N)$ is differentiable at 0 with $\Psi(0)=I,\ \Psi'(0)=V$,
where $W^{1,\infty}(\mathbb{R}^N,\mathbb{R}^N)$ is the set of bounded Lipschitz maps from $\mathbb{R}^N$ into itself, $I$ is the identity and $V$ is a given deformation field. 

We will use the notation 
$\Omega_t=\Psi(t)(\Omega)$, for a given set $\Omega$, $\lambda_n(t):=\lambda_n(\Omega_t,\alpha)$, $u_t$ is an associated eigenfunction with unitary $L^2$ norm, and $u'$ will
denote the derivative of $u_t$ at $t=0$. Moreover, we assume that $\lambda_n(0)$ is simple.

It is well known (see e.g., \cite{dz}) that if we define 
\[J(t)=\int_{\Omega_t} y(t,x)dx,\]
for some function $y$, then the Hadamard shape derivative is given by
\begin{equation}
\label{hadderfun}
J'(0)=\int_\Omega \fr{\partial y}{\partial t}(0,x)dx+\int_{\partial \Omega}y(0,x)V\cdot\nu \,ds_x.
\end{equation}

As a consequence we have
\begin{theorem}
\label{hadamard}
Let $\Omega$ be a bounded open set of class $C^4$. The Hadamard shape derivative for a simple eigenvalue $\lambda$ of problem~\eqref{dirichlet} with corresponding eigenfunction $u$ is given by 
\begin{equation}
\label{evderiv}
\lambda'(0)=-\int_{\partial \Omega}\left(\fr{\partial^2u}{\partial \nu^2}\right)^2V\cdot\nu\, ds_x.
\end{equation}
\end{theorem}
\noindent
\begin{proof}
We have
\begin{equation}
\label{lambdat}
\lambda(t)=\int_{\Omega_t}(\Delta u_t)^2-\alpha\left|\nabla u_t\right|^2dx
\end{equation}
and the eigenfunction is normalized,
\begin{equation}
\label{normalization}
\int_{\Omega_t}u_t^2dx=1.
\end{equation}

The function $u'$ can be calculated by solving the following boundary value problem (c.f.~\cite{Grinf,Henrot})
\begin{equation}\label{up}
\left\{\begin{array}{ll} \Delta^2 u'+\alpha\Delta u'=\lambda' u+\lambda u', &  \text{in } \Omega,\\[1mm]
u'=0 , &  \text{on } \partial\Omega\\[1mm]
\fr{\partial u'}{\partial \nu }=-\fr{\partial^2u}{\partial \nu^2}\left(V\cdot\nu\right) , &  \text{on } \partial\Omega\\[1mm]
\int_{\Omega}uu'dx=0, &  \end{array}\right.
\end{equation}


Since the case $\alpha=0$ can be recovered from \cite{ortzua} (and can be done similarly to what follows), we assume $\alpha\neq0$ and the eigenvalue equation can be written as
\[\Delta u=\fr{\lambda u}{\alpha}-\fr{\Delta^2 u}{\alpha},\]
so that we have
\begin{align*}
\int_\Omega\nabla u\nabla u'dx&=\int_{\partial\Omega}u'\ \fr{\partial u}{\partial \nu }ds_x-\int_\Omega u'\ \Delta u\, dx\\
&=-\int_\Omega u'\ \left(\fr{\lambda u}{\alpha}-\fr{\Delta^2 u}{\alpha}\right)dx\\
&=\fr{1}{\alpha}\left(\int_{\partial\Omega}u'\ \fr{\partial(\Delta u)}{\partial \nu }ds_x-\int_\Omega\nabla u'\nabla\left(\Delta u\right)\right)dx\\
&=-\frac{1}{\alpha}\left(\int_{\partial\Omega}\Delta u\ \frac{\partial u'}{\partial \nu }ds_x-\int_\Omega\Delta u\ \Delta u'dx\right)\\
&=-\frac{1}{\alpha}\int_{\partial\Omega}\Delta u\ \left(-\frac{\partial^2u}{\partial \nu ^2}\right)V\cdot\nu\, ds_x+\frac{1}{\alpha}\int_\Omega\Delta u\ \Delta u'dx\\
&=\frac{1}{\alpha}\int_{\partial\Omega}\Delta u\ \left(\frac{\partial^2u}{\partial \nu ^2}\right)V\cdot\nu \,ds_x+\frac{1}{\alpha}\int_\Omega\Delta u\ \Delta u'dx.\numberthis \label{eqn2}
\end{align*}

Applying now formula~\eqref{hadderfun} to the equation~\eqref{lambdat} and using~\eqref{eqn2} we obtain
\[
\begin{array}{lll}
\lambda'(0) & = & 2\dint_{\Omega}\Delta u\Delta u'-\alpha\nabla u\nabla u'dx+\dint_{\partial \Omega}(\Delta u)^2V\cdot\nu \,ds_x\eqskip
& = & 2\dint_{\Omega}\Delta u\Delta u'dx-2\alpha\dint_\Omega\nabla u\nabla u'dx+\dint_{\partial \Omega}(\Delta u)^2V\cdot\nu \,ds_x\eqskip
& = & 2\dint_{\Omega}\Delta u\Delta u'dx+2\dint_{\partial\Omega}\Delta u\ \left(-\frac{\partial^2u}{\partial \nu ^2}\right)V\cdot\nu\, ds_x\eqskip
&  & \hspace*{1cm}-2\dint_\Omega\Delta u\ \Delta u'dx+\dint_{\partial \Omega}(\Delta u)^2V\cdot\nu \,ds_x\eqskip
& = & \dint_{\partial \Omega}\left(-2\frac{\partial^2u}{\partial \nu ^2}\ \Delta u+(\Delta u)^2\right)V\cdot\nu\, ds_x.
\end{array}
\]
The proof is concluded once we observe that $u\in H^4(\Omega)$ (cf.\ \cite{ggs}), and since $u=\frac{\partial u}{\partial \nu}=0$ on $\partial \Omega$, we have that
\[
\Delta u=\frac{\partial^2u}{\partial\nu^2}\ {\rm on\ }\partial\Omega.
\]

\end{proof}

\begin{remark}
Using formula \eqref{evderiv}, we may try to attack Problem \ref{shapeoptdir} via the Lagrange Multiplier Theorem. Since the constraint here is $|\Omega|=1$, we obtain the following condition
\begin{equation}
\label{critic}
\frac{\partial^2u}{\partial\nu^2}={\rm constant}\ {\rm on\ }\partial\Omega.
\end{equation}
Note that condition \eqref{critic} has then to be added to problem \eqref{dirichlet}, yielding an overdetermined problem resembling the Serrin problem (see \cite{serrin}). However, problem \eqref{dirichlet} coupled with condition \eqref{critic} is a more difficult problem, and the only partial result available in the literature can be found in \cite{dalmasso}.

It is worth observing that solving the overdetermined problem \eqref{dirichlet}, \eqref{critic} is not equivalent to solving Problem \ref{shapeoptdir}: in fact, the former provides just a critical point, that may be only a local minimizer, or even a local maximizer. Interestingly enough, though, eigenfunctions on the ball always satisfy condition \eqref{critic}. For a more detailed analysis of this fact, we refer to \cite{buososurvey, bl}.
\end{remark}


\section{Numerical Methods}
\subsection{Numerical solution of the eigenvalue problem}
\label{nummeth}
In this section we will describe a numerical method for solving~\eqref{dirichlet}.

A fundamental solution $\Phi_{\lambda}$ of the partial differential equation of the eigenvalue problem~\eqref{dirichlet} is given by (see e.g., \cite{Kita})
\begin{equation}
\Phi_{\lambda}(x)=\frac{i\left(H_0^{(1)}\left(i\sqrt{\fr{1}{2}(\sqrt{\alpha^2+4\lambda}-\alpha)}|x|\right)-H_0^{(1)}\left(\sqrt{\fr{1}{2}(\sqrt{\alpha^2+4\lambda}+\alpha)}|x|\right)\right)}{4\sqrt{\alpha^2+4\lambda}},
\end{equation}
where $H_{0}^{(1)}$ is a Hankel function of the first kind. 

We will consider particular solutions of the partial differential equation of the eigenvalue problem~\eqref{dirichlet}, by defining the boundary integral operators $(x\in\Omega)$
\[
u(x)=\dint_{\hat{\Gamma}}\Phi_{\lambda}(x-y)\,\varphi(y)ds_{y}+\int_{\hat
{\Gamma}}\partial_{\nu _{y}}\Phi_{\lambda}(x-y)\psi(y)\,ds_{y},
\]
where $\hat{\Gamma}$ is an artificial boundary that surrounds $\partial\Omega$ (see e.g., \cite{AA,Ant}), and $\varphi$ and $\psi$ are densities. The numerical approximation of an arbitrary solution of the PDE of the eigenvalue problem~\eqref{dirichlet} can be justified by density results e.g., \cite{AA,Ant2}. Moreover, we will assume that $\hat{\Gamma}$ does not intersect $\bar{\Omega}$. Thus, we can discretise the boundary integral operators by considering the linear combinations

\begin{equation}
u_{m}(x)=\dsum_{j=1}^{m}\alpha_{m,j}\Phi_{\lambda}(x-y_{m,j})+\dsum_{j=1}%
^{m}\beta_{m,j}\partial_{\nu _{y_{m,j}}}\Phi_{\lambda}(x-y_{m,j}),
\label{aproximacaoMFS}%
\end{equation}
where $y_{m,j}$ are some points on $\hat{\Gamma}.$ Note that the functions $u_{m}$ are particular solutions of the partial differential equation involved in the eigenvalue problem~\eqref{dirichlet} and the coefficients can be determined by fitting the boundary conditions of the problem.


We consider some collocation points $x_{1},...,x_{m}$, (almost) uniformly distributed on $\partial\Omega$ and impose the boundary conditions of the problem which leads to the $(2m)\times(2m)$ system
\begin{equation}
\left\{
\begin{array}
[c]{l}%
0=u_{m}(x_{i})=\dsum_{j=1}^{m}\alpha_{m,j}\Phi_{\lambda}(x_{i}-y_{m,j}%
)+\dsum_{j=1}^{m}\beta_{m,j}\partial_{\nu _{y_{m,j}}}\Phi_{\lambda}(x_{i}%
-y_{m,j}),\\
0=\partial_{\nu _{x_{i}}}u_{m}(x_{i})=\dsum_{j=1}^{m}\alpha_{m,j}\partial
_{\nu _{x_{i}}}\Phi_{\lambda}(x_{i}-y_{m,j})+\dsum_{j=1}^{m}\beta_{m,j}%
\partial_{\nu _{x_{i}}}\partial_{\nu _{y_{m,j}}}\Phi_{\lambda}(x_{i}-y_{m,j}).
\end{array}
\right.  \label{sistemapl}%
\end{equation}

We will consider the choice for source points $y_{m,j}$ described in~\cite{AA}, assume that $\nu _{y_{m,j}}=\nu _{x_{j}}$, and denote this vector simply by
$\nu _{j}$. Using the notation $d_{i,j}=x_{i}-y_{m,j}$, the system~\eqref{sistemapl} can be rewritten as
\begin{equation}
\left\{
\begin{array}
[c]{l}%
0=\dsum_{j=1}^{m}\alpha_{m,j}\Phi_{\lambda}(d_{i,j})+\dsum_{j=1}^{m}\beta
_{m,j}\left(  \nu _{j}\cdot\nabla\Phi_{\lambda}(d_{i,j})\right), \\
0=\dsum_{j=1}^{m}\alpha_{m,j}\left(  \nu _{i}\cdot\nabla\Phi_{\lambda}%
(d_{i,j})\right)  +\dsum_{j=1}^{m}\beta_{m,j}\left(  \nu _{i}\cdot\nabla\left(
\nu _{j}\cdot\nabla\Phi_{\lambda}(d_{i,j})\right)  \right).
\end{array}
\right.  \label{sistema2}%
\end{equation}

The approximations of the eigenvalues can be calculated by adapting the Betcke-Trefethen method (see \cite{bt}) to this context. We consider $p$ points $z_1,z_2,...,z_p$, randomly chosen in $\Omega$ and define the following six blocks
\begin{center}
$%
\begin{array}
[c]{lll}%
A(\lambda)=\left[  \Phi_{\lambda}(d_{i,j})\right]  _{m\times m}, &  &
B(\lambda)=\left[  \nu _{j}\cdot\nabla\Phi_{\lambda}(d_{i,j})\right]  _{m\times
m},\\
&  & \\
C(\lambda)=\left[  \nu _{i}\cdot\nabla\Phi_{\lambda}(d_{i,j})\right]  _{m\times
m}, &  & D(\lambda)=\left[  \nu _{i}\cdot\nabla\left(  \nu _{j}\cdot\nabla
\Phi_{\lambda}(d_{i,j})\right)  \right]  _{m\times m},\\
&  & \\
E(\lambda)=\left[  \Phi_{\lambda}(\tilde{d}_{i,j})\right]  _{p\times m}, &  &
F(\lambda)=\left[  \nu _{j}\cdot\nabla\Phi_{\lambda}(\tilde{d}_{i,j})\right]  _{p\times
m},
\end{array}
$
\end{center}
where $\tilde{d}_{i,j}=z_{i}-y_{m,j}$. Then, we define the matrix

\[\mathbf{M}(\lambda)=\left[
\begin{array}
[c]{cc}%
A(\lambda)
 & B(\lambda)\\
C(\lambda)
 & D(\lambda)\\
E(\lambda)
 & F(\lambda)\end{array}\right], \]
\noindent
compute the $\mathbf{QR}$ decomposition of $\mathbf{M}(\lambda)$, and calculate the minimal eigenvalue of the first $(2m)\times(2m)$ block of the matrix $\mathbf{M}(\lambda)$
that will be denoted by $\sigma_1(\lambda)$. The approximations for the eigenvalues of problem~\eqref{dirichlet} are the values $\lambda$, for which $\sigma_1(\lambda)\approx0$.

\subsection{Numerical shape optimization}

In this section we will consider the shape optimization Problem~\ref{shapeoptdir} among general simply connected planar domains, whose boundary can
be parametrized by \[\partial\Omega=\left\{(\Gamma_1(t),\Gamma_2(t)):t\in[0,2\pi[\right\},\]
for some continuous and $(2\pi)$-periodic functions $\Gamma_1$ and $\Gamma_2$. We will consider the (truncated) Fourier expansions

\[\Gamma_1(t)\approx\gamma_1(t)=\dsum_{j=0}^Pa_j^{(1)}\cos(jt)+\dsum_{j=1}^Pb_j^{(1)}\sin(jt)\]
and
\[\Gamma_2(t)\approx\gamma_2(t)=\dsum_{j=0}^Pa_j^{(2)}\cos(jt)+\dsum_{j=1}^Pb_j^{(2)}\sin(jt),\]
\noindent
for a sufficiently large $P\in\mathbb{N}$, and the optimization procedure consists in finding optimal coefficients $a_j^{(1)}$, $b_j^{(1)}$, $a_j^{(2)}$, $b_j^{(2)}$. The optimization is performed by a gradient-type method, using the Hadamard shape derivative given by Theorem~\ref{hadamard} to calculate the derivative of the eigenvalue with respect to perturbations of the coefficients $a_j^{(1)}$, $b_j^{(1)}$, $a_j^{(2)}$, $b_j^{(2)}$.

\subsection{Numerical results}
\label{numres}
In this section we present the main results that we gathered with our numerical procedure for solving Problem~\ref{shapeoptdir}. 

As was mentioned in the Introduction, each of the eigenvalue curves $\lambda_k(\alpha)$ is made up of analytic eigenvalue branches which intersect each other. We illustrate this fact in Figure~\ref{fig:figurasde}. As was shown in Theorem~\ref{mainthm}, all the eigenvalues have the following asymptotic behaviour
\[\lambda_k(\Omega,\alpha)=-\fr{\alpha^2}{4} + \so(\alpha^2).\]

Thus, in order to produce more convenient pictures, instead of plotting the first eigenvalues as functions of $\alpha$, we will extract the first term of the expansion, which is the same for all eigenvalues, i.e., in Figure~\ref{fig:figurasde} we plot the quantities 
\[\lambda_k(\Omega,\alpha)+\fr{\alpha^2}{4},\; k=1,2,...,10,\]
for a disk of unit area and similar results for an ellipse with unit area and eccentricity equal to $\sqrt{3}/2$.


Figure~\ref{fig:lambdaopt} shows the curve of the quantity $\lambda_1^\ast(\alpha)+\frac{\alpha^2}{4}$. We can observe several branches corresponding to different types of
minimizers. Some of them, obtained for $\alpha=110,170,230,400$, are plotted in Figure~\ref{fig:figuras2}. The optimal eigenvalue $\lambda_1^\ast(\alpha)$ is the minimum
among the values obtained for all the branches. We calculated the critical value of $\alpha$, which is the maximal value of $\alpha$ for which the ball is the minimizer
and obtained $\alpha^\star\approx 102.23$. In \cite{ABM} it was proved that the ball is the minimizer for $\alpha\in[0,a],$ for some $a<\Lambda$, where
$\Lambda=\pi j_{1,1}^2\approx12.0377$ is the first buckling eigenvalue of the disk with unit area. Our numerical results suggest that actually the result may be true for
a larger range of values of $\alpha$ and we conjecture that the ball is the minimizer for $\alpha\in[0,\alpha^\star]$. On the other hand, we have numerical evidence to support the conjecture that for $\alpha>\alpha^\star$, the ball is no longer the minimizer. For instance, for $\alpha=110$, the first eigenvalue of the ball of unit area can be directly calculated by solving \eqref{negative_lambda} and is equal to -1622.16613... In Table~\ref{tab1} we show some numerical approximations for the first eigenvalue of the minimizer that we obtained with our algorithm when $\alpha=110$, which is plotted in Figure~\ref{fig:figuras2}, for different values of $m$. These results suggest that the first eigenvalue of this domain is equal to -1786.35377..., which is significantly smaller than the first eigenvalue of the disk.

\begin{table}[ht]
\begin{center}
\begin{tabular}{|c|c|}
\hline $m$     & $\tilde{\lambda}_1$  \\
\hline 1000   &  -1786.3537774 \\
\hline 1500   & -1786.3537779 \\
\hline 1800   & -1786.3537762 \\
\hline 2000   & -1786.3537753 \\

\hline
\end{tabular}
\end{center}
\vspace*{0.4cm} \caption{\textit{Numerical approximations obtained for the first eigenvalue of the minimizer when $\alpha=110$, for different values of $m$.}}
\label{tab1}
\end{table}

\begin{figure}[ht]
\centering
\includegraphics[width=0.6\textwidth]{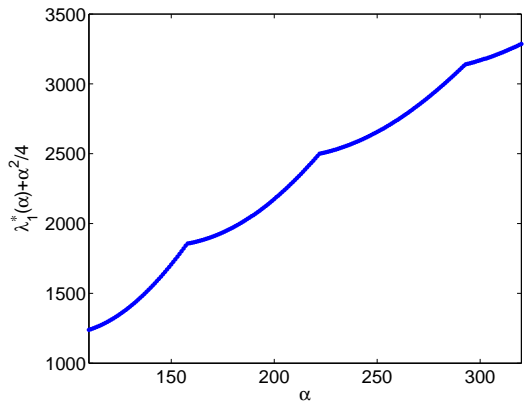}
\caption{The quantity $\lambda_1^\ast(\alpha)+\fr{\alpha^2}{4}$, for $\alpha\in[110,320]$.}
\label{fig:lambdaopt}
\end{figure}

In Figure~\ref{fig:figuras3} we plot the eigenfunctions associated to the first three eigenvalues of the optimizers of $\lambda_1$, obtained for $\alpha=110,170,230$. In
this work we considered just the optimization of the first eigenvalue. However, we observed that, besides the fact that the eigenfunction associated with the first eigenvalue
changes sign, it also has different number of nodal domains, depending on the parameter $\alpha$. Moreover, 'similar' eigenfunctions appear associated with eigenvalues of
different orders. For instance, the eigenfunction associated with the first eigenvalue for $\alpha=110$ is antisymmetric with respect to the first axis. However, the eigenfunction
associated with the first eigenvalue for $\alpha=170$ is symmetric with respect to the first axis and the first antisymmetric eigenfunction with respect to the first
axis is associated not with the first eigenvalue, but with the second eigenvalue.

\begin{figure}[ht]
\centering
\includegraphics[width=0.48\textwidth]{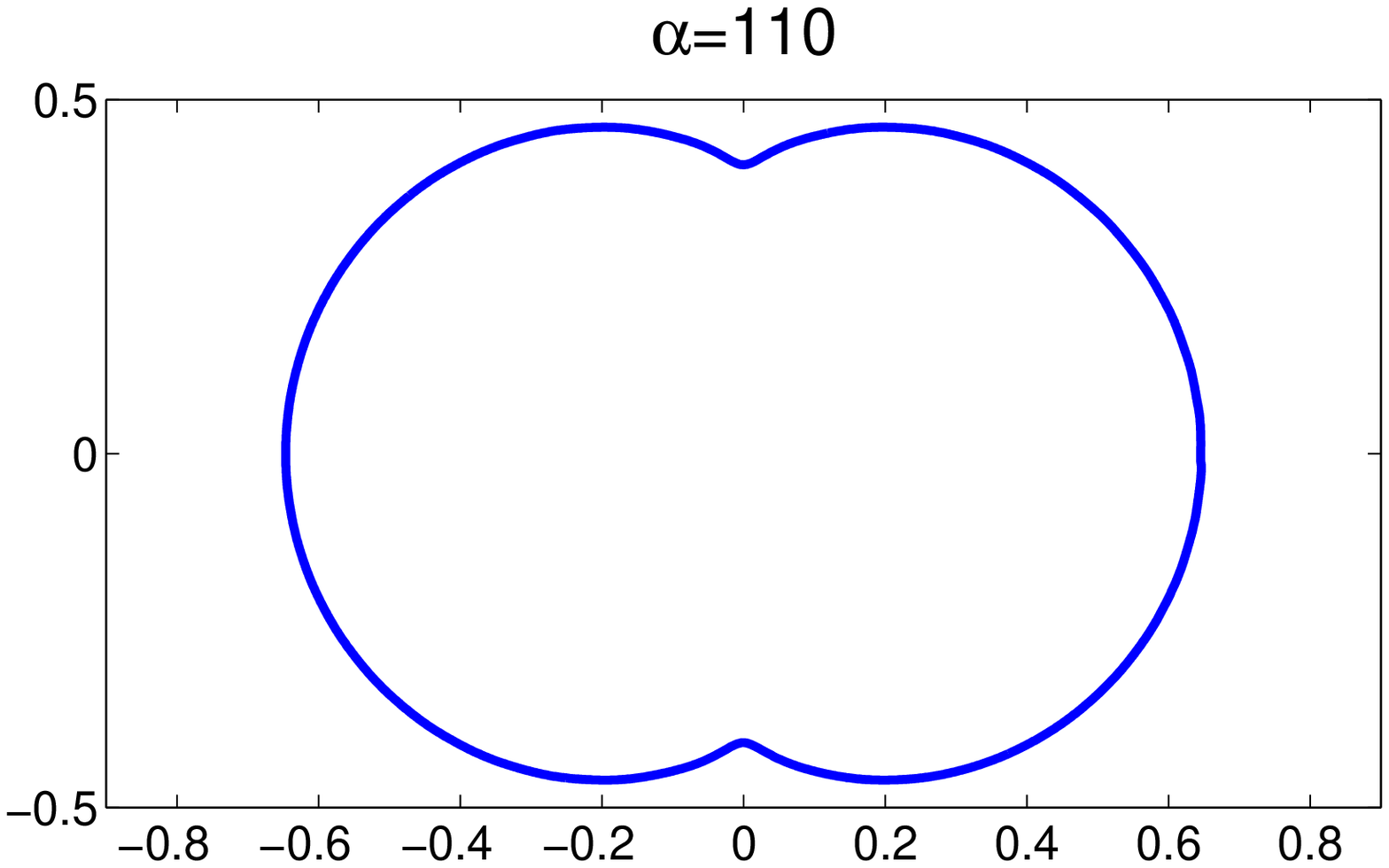}
\includegraphics[width=0.48\textwidth]{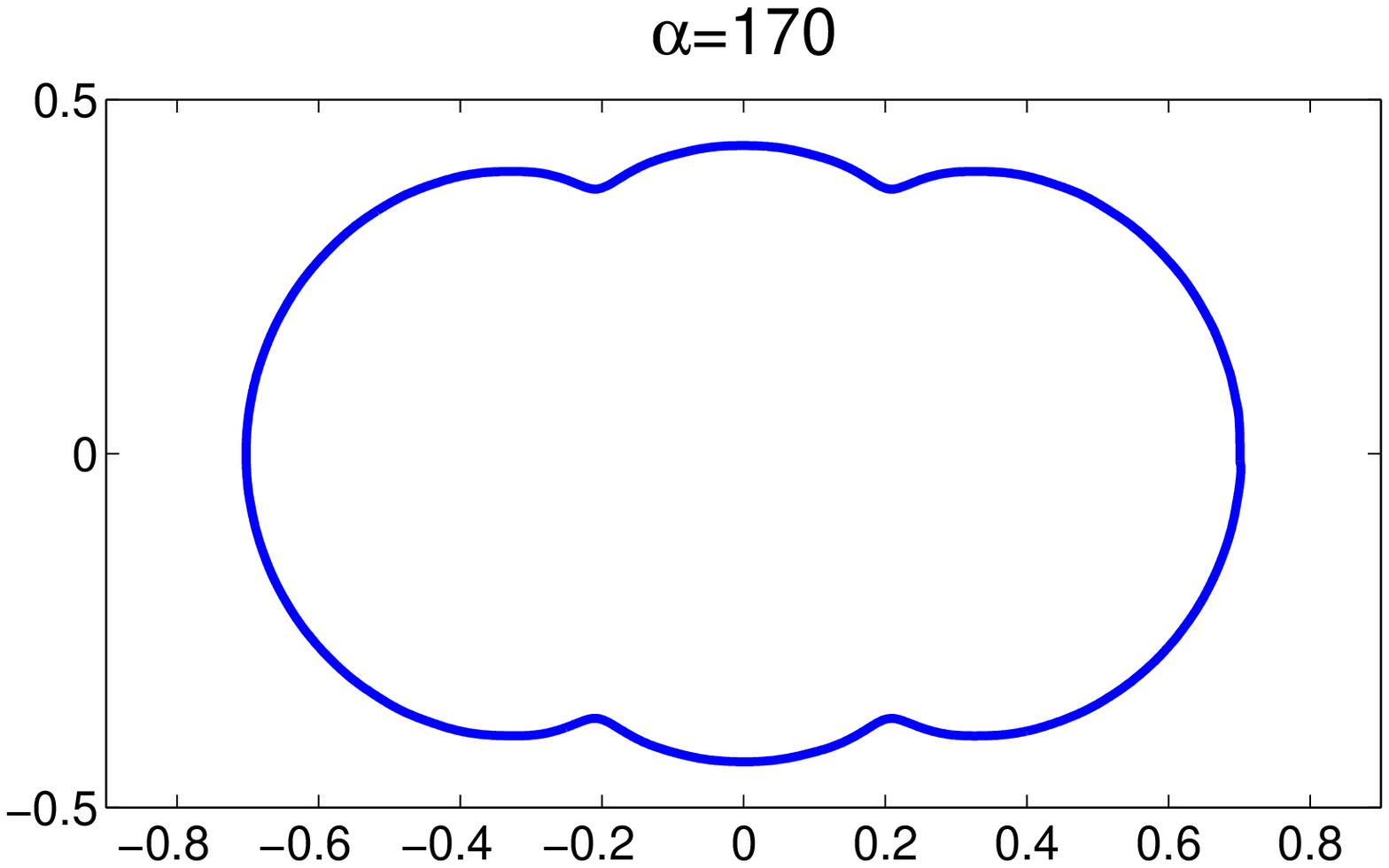}
\includegraphics[width=0.48\textwidth]{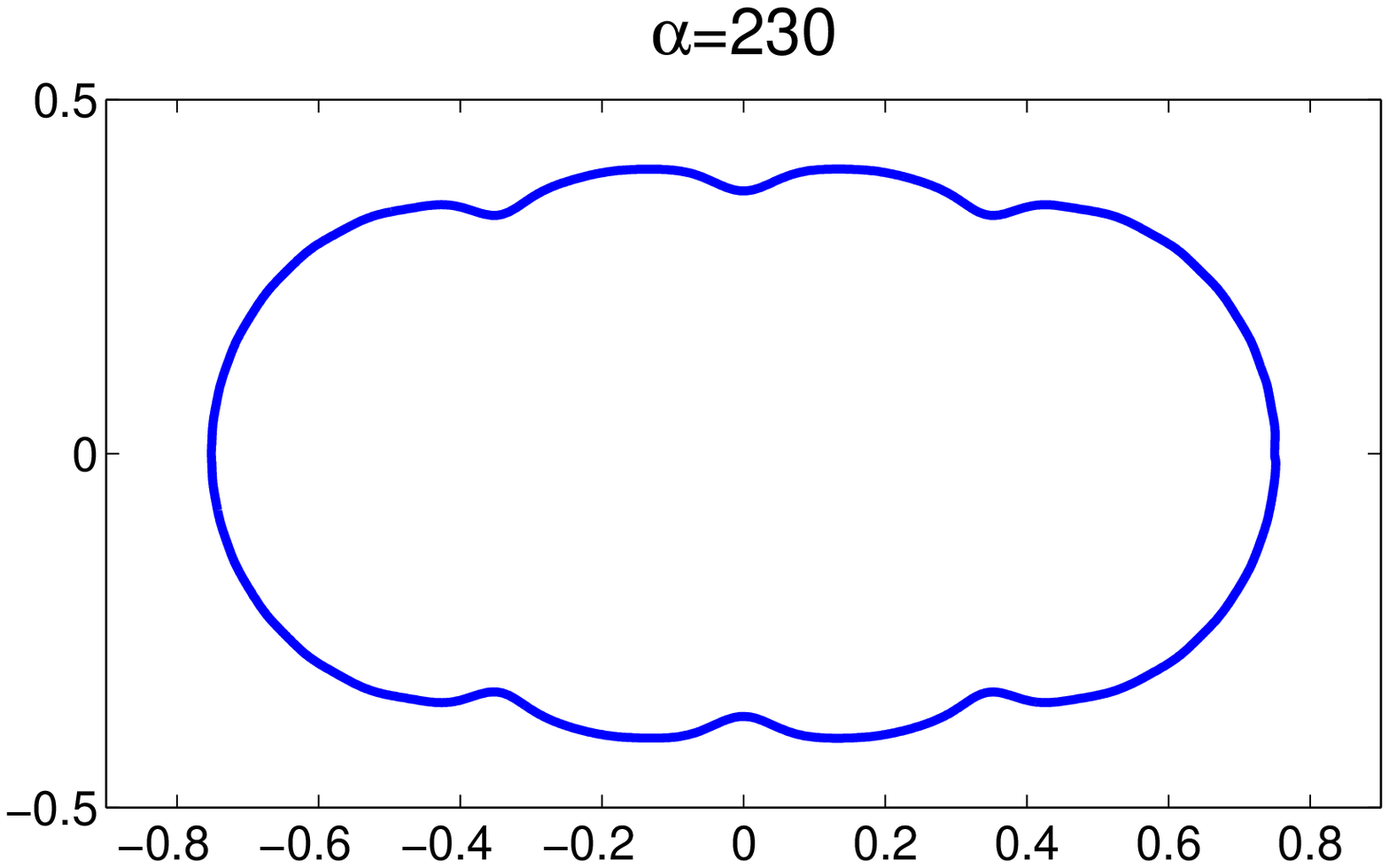}
\includegraphics[width=0.48\textwidth]{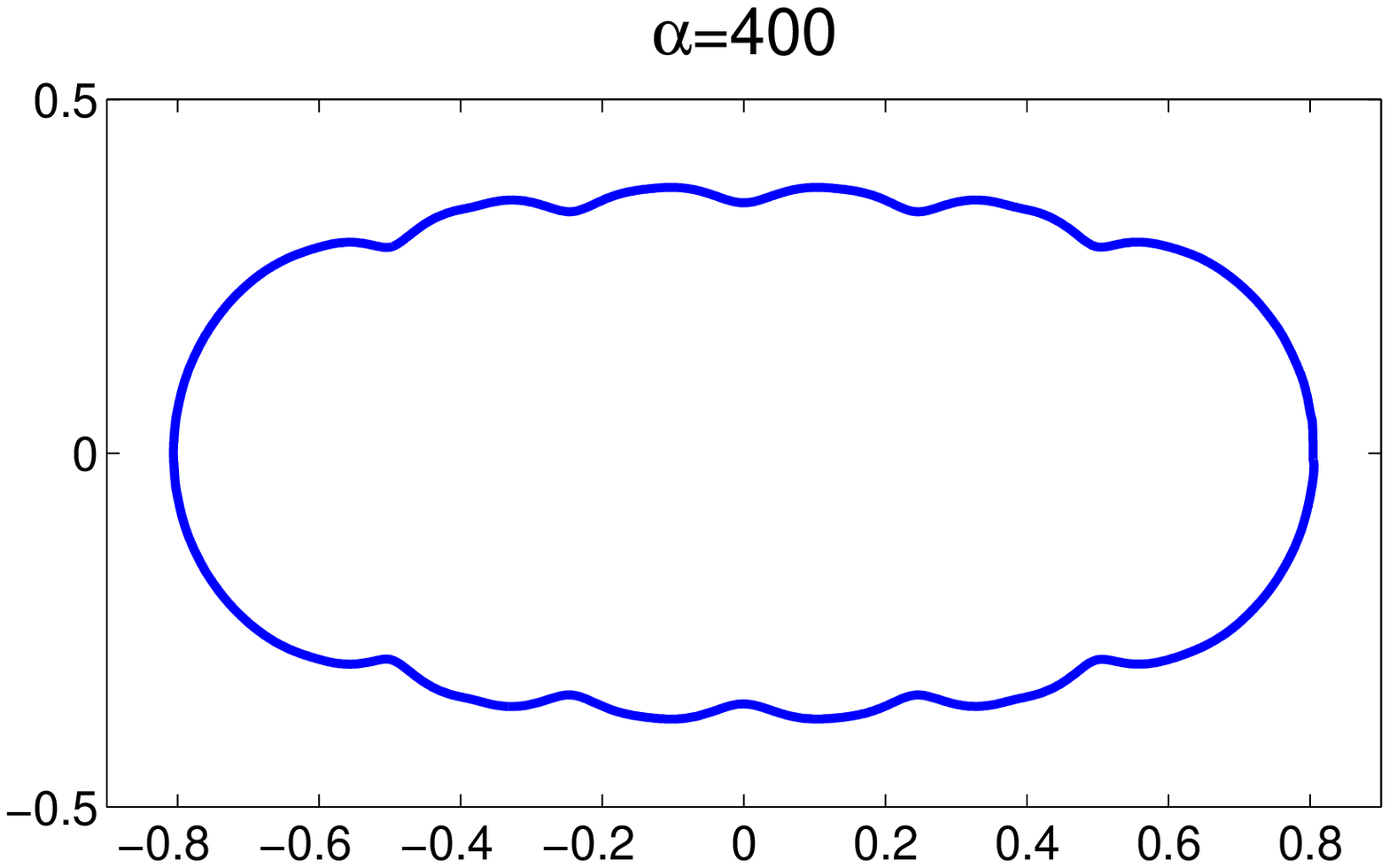}
\caption{Minimizers of $\lambda_1(\alpha)$, for $\alpha=110,170,230,400$.}
\label{fig:figuras2}
\end{figure}

Figure~\ref{fig:reentr} shows a zoom of the boundary of the optimizer obtained numerically for $\alpha=110$, in a neighbourhood of the re-entrant
part of the boundary. Note that the boundary of the domains considered in the optimization procedure was parameterized by a (truncated) Fourier expansion. In
particular the domains considered are always smooth and it is not clear how to obtain information on the regularity of the boundary of the optimizer from this.
In particular, it is not possible to deduce whether this corresponds to a smooth boundary, a corner, or even a cusp.
\begin{figure}[ht]
\centering
\includegraphics[width=0.48\textwidth]{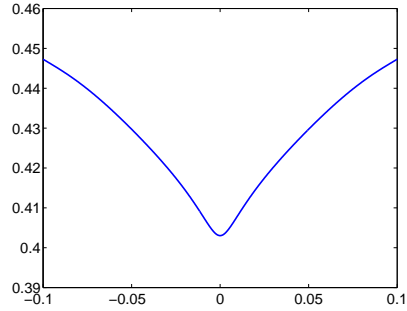}
\caption{Zoom of the boundary of the optimizer obtained for $\alpha=110$ close to the re-entrant region.}
\label{fig:reentr}
\end{figure}

\begin{figure}[ht]
\centering
\includegraphics[width=0.325\textwidth]{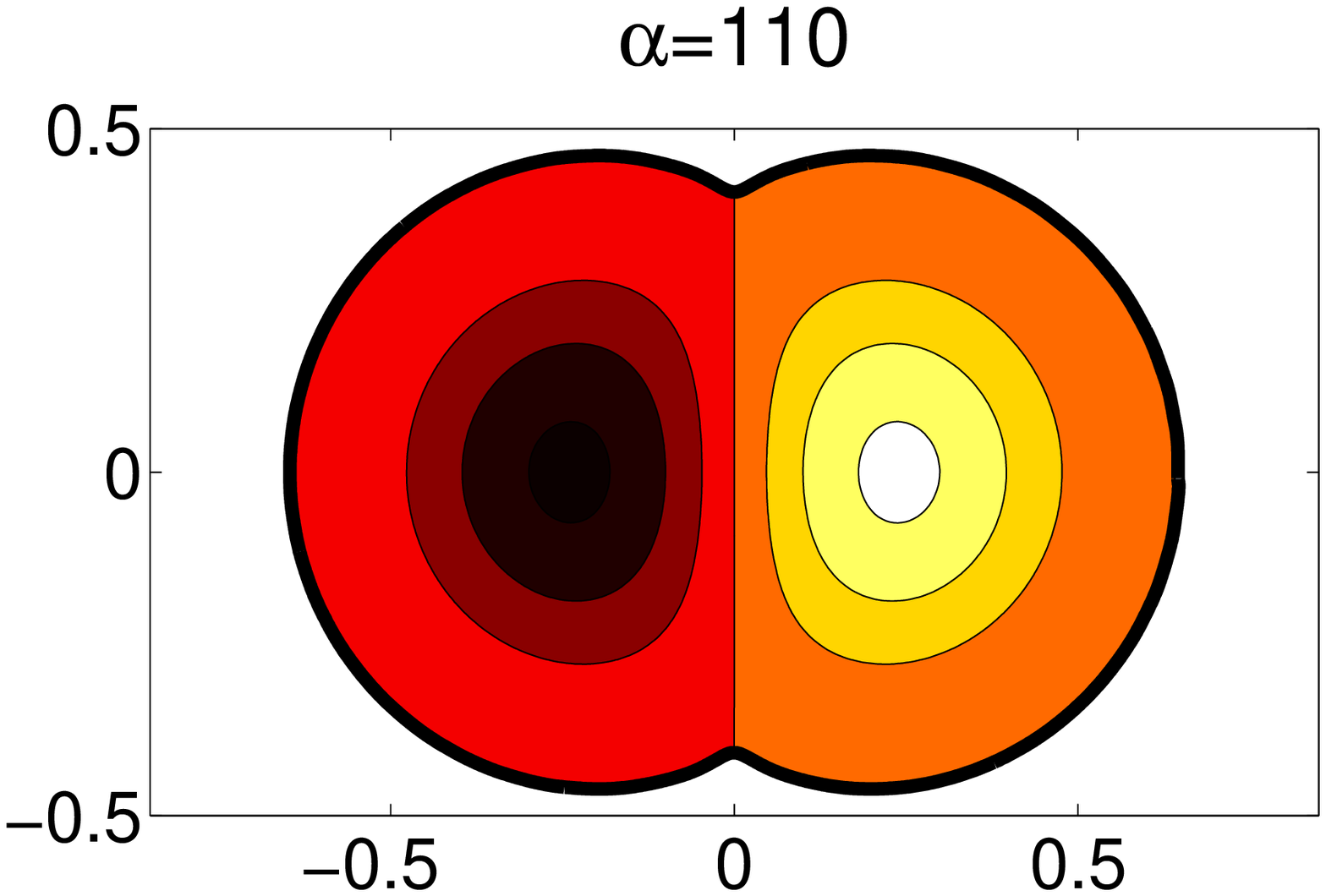}
\includegraphics[width=0.325\textwidth]{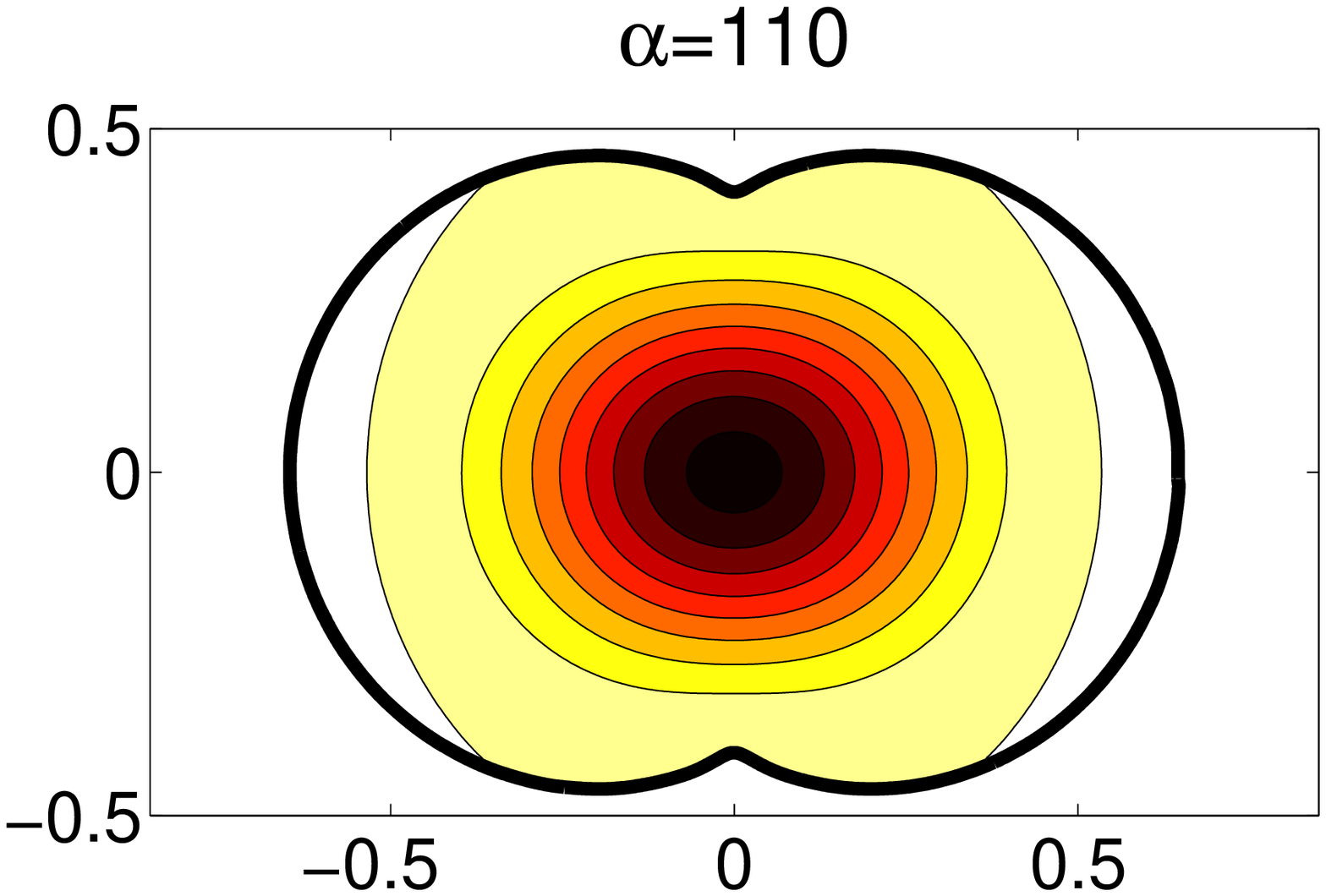}
\includegraphics[width=0.325\textwidth]{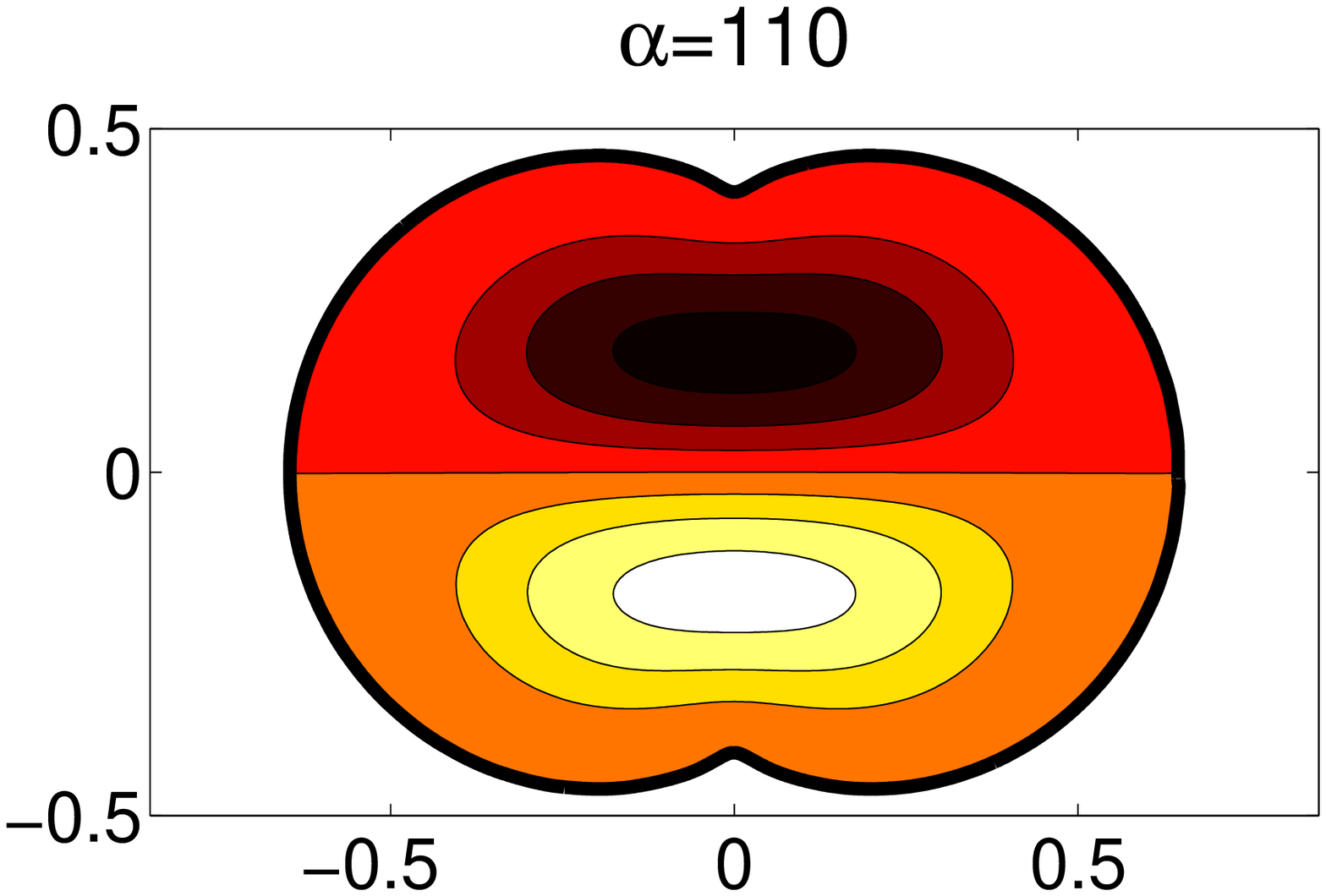}
\includegraphics[width=0.325\textwidth]{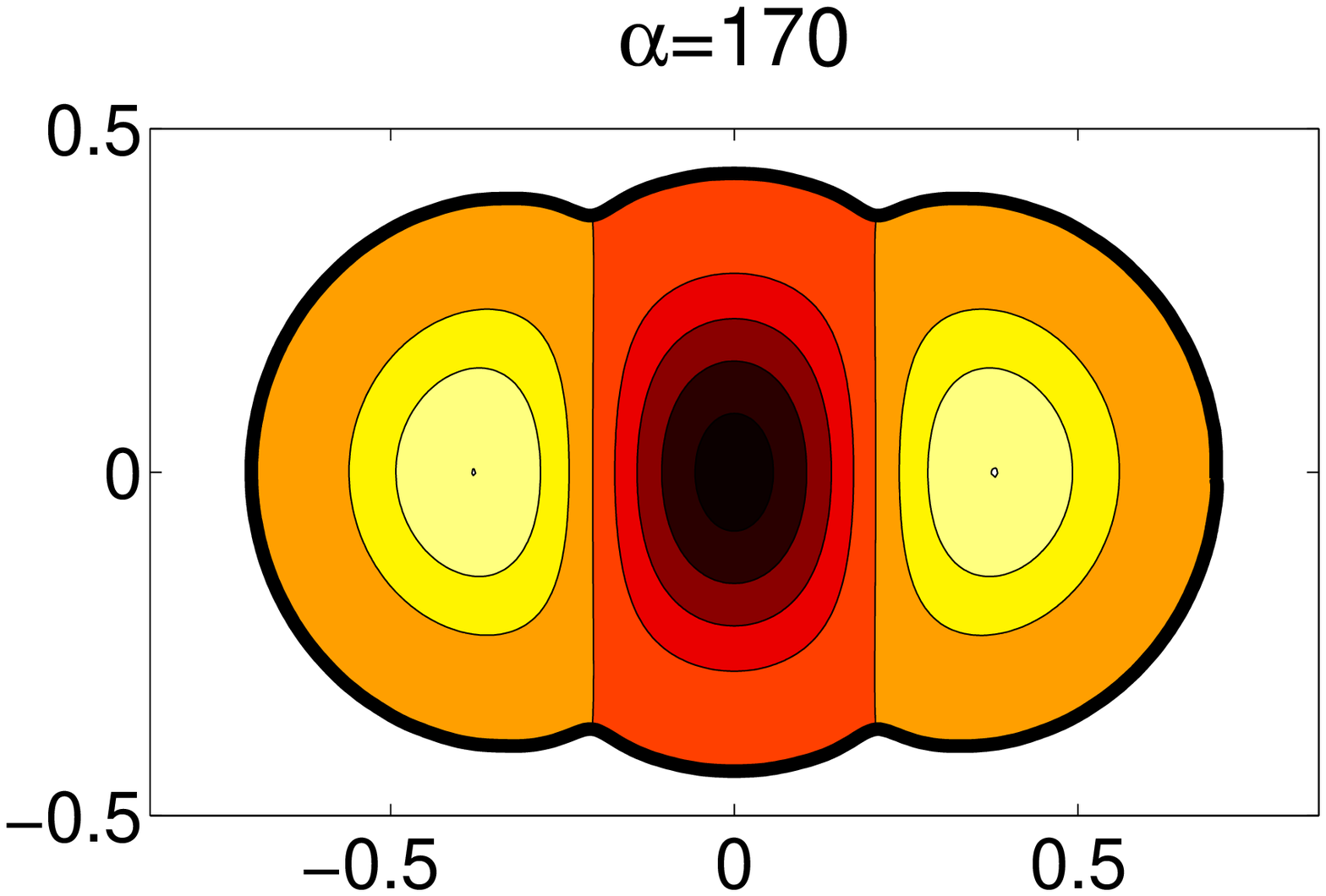}
\includegraphics[width=0.325\textwidth]{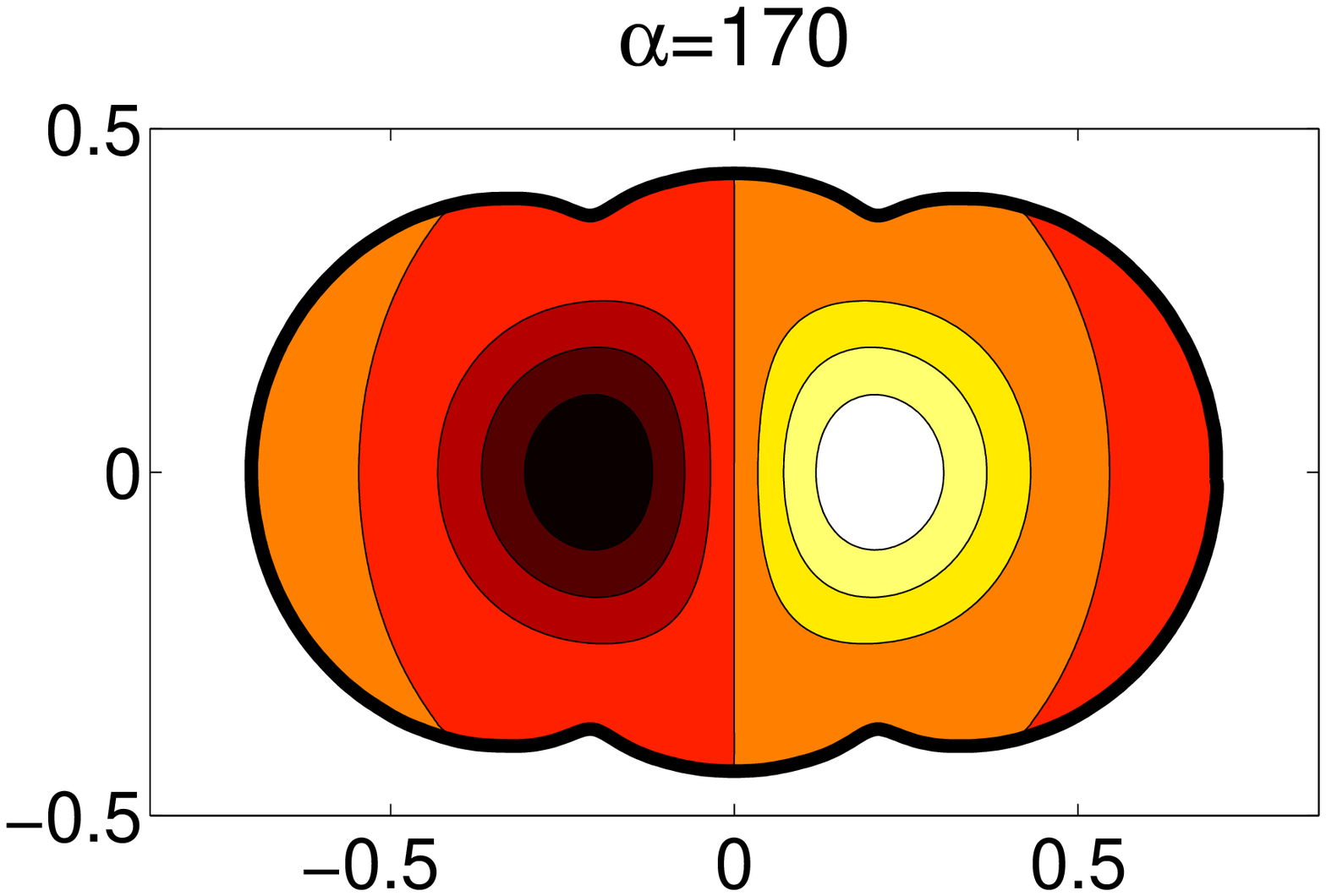}
\includegraphics[width=0.325\textwidth]{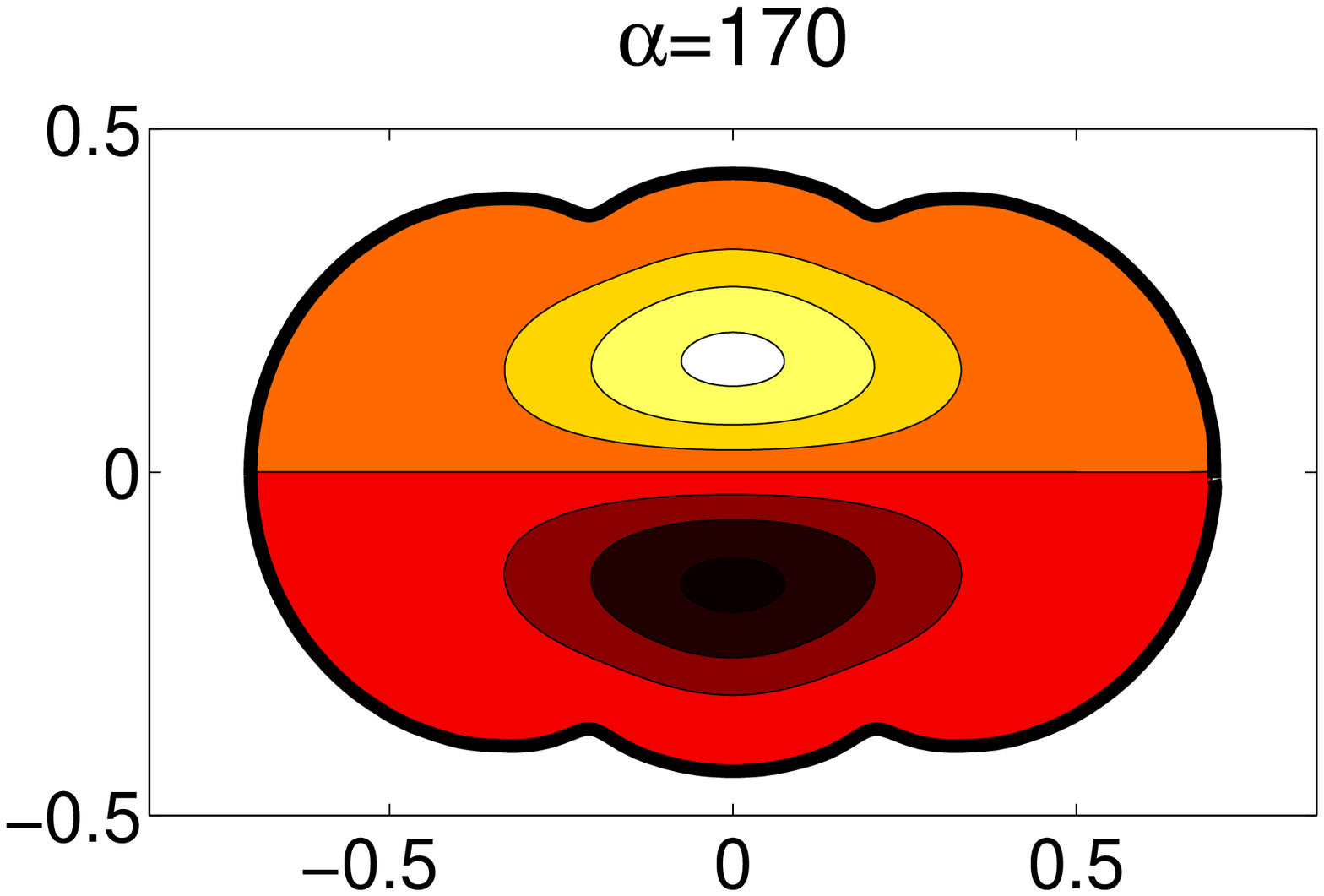}
\includegraphics[width=0.325\textwidth]{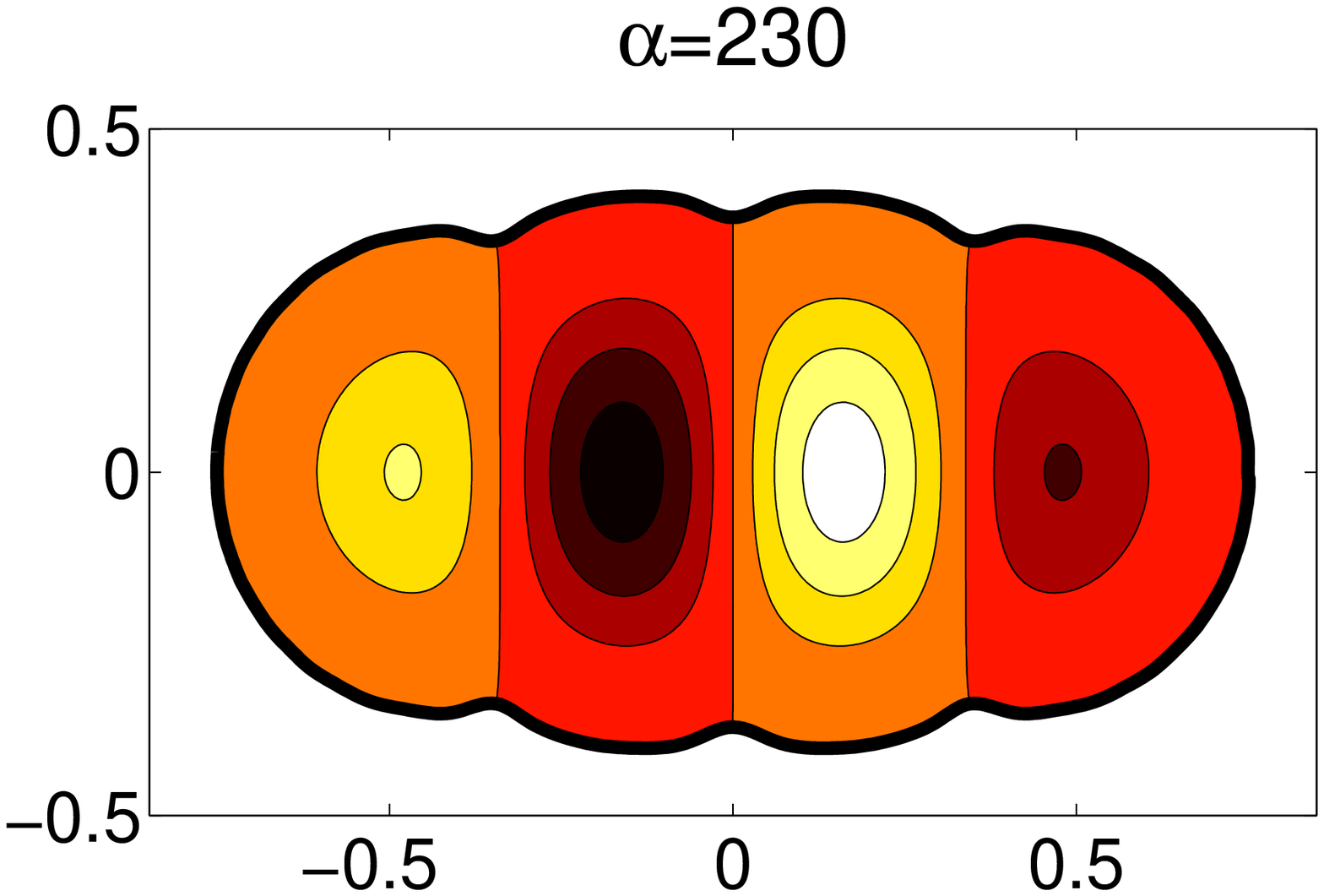}
\includegraphics[width=0.325\textwidth]{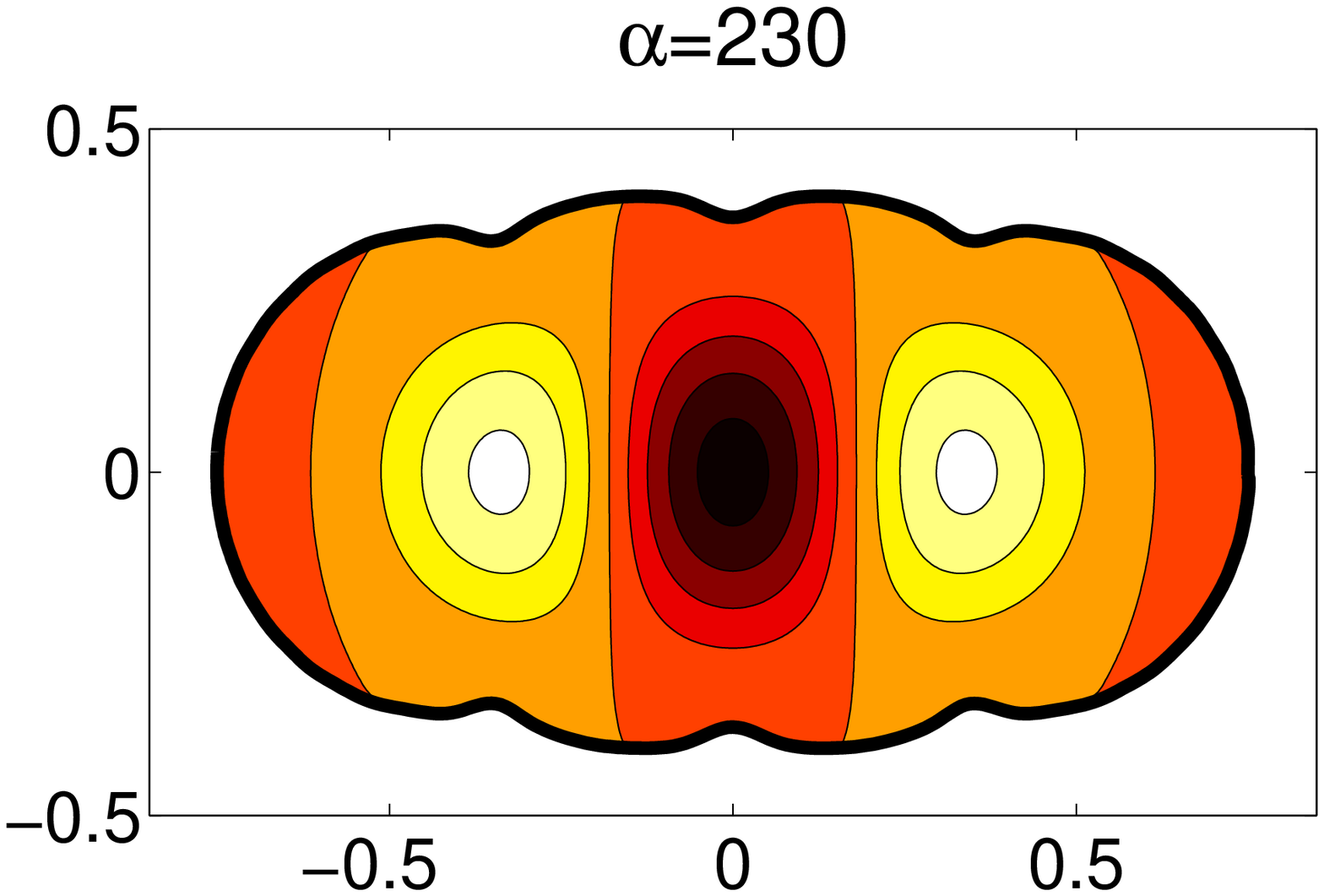}
\includegraphics[width=0.325\textwidth]{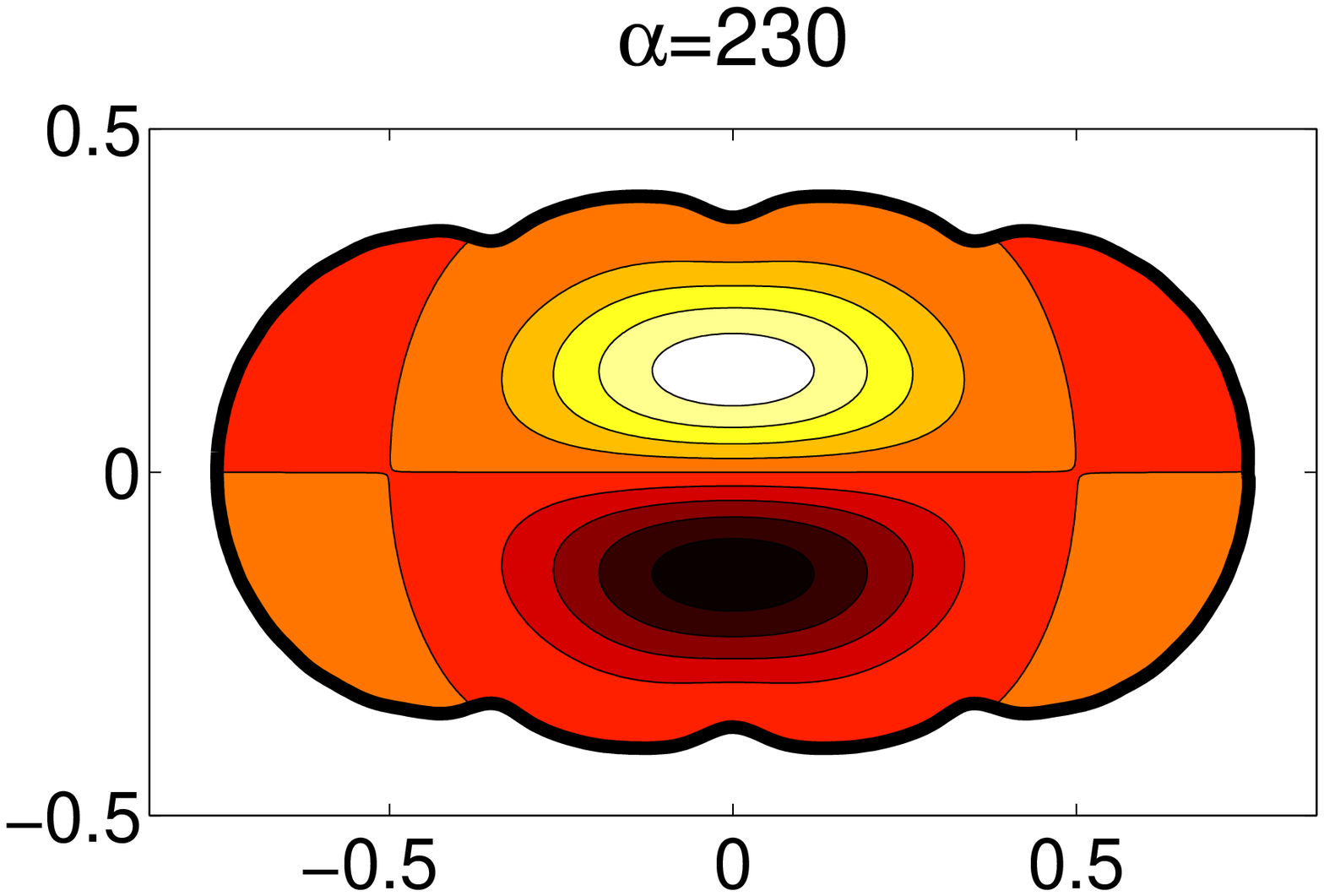}
\caption{Plots of the eigenfunctions associated to the first three eigenvalues of the optimizers of $\lambda_1$, obtained for $\alpha=110,170,230$.}
\label{fig:figuras3}
\end{figure}


\section*{Acknowledgements}
This work was partially supported by the Funda\c c\~{a}o para a Ci\^{e}ncia e a Tecnologia
(Portugal) through the program ``Investigador FCT'' with reference IF/00177/2013 and the project {\it Extremal spectral quantities and related problems}
(PTDC/MAT-CAL/4334/2014).
Most of the research in this paper was carried out while the second author held a post-doctoral position at the University of Lisbon
within the scope of this project. The second author is a member of the Gruppo Nazionale per l'Analisi
Matematica, la Probabilit\`a e le loro Applicazioni (GNAMPA) of the Istituto Naziona\-le di Alta Matematica (INdAM).



\begin{thebibliography}{10}

\bibitem{AA}
{\sc C.~J.~S. Alves and P.~R.~S. Antunes}, {\em The {M}ethod of {F}undamental
  {S}olutions applied to the calculation of eigensolutions for 2{D} plates},
  SIAM Journal on Matrix Analysis and Applications, 77 (2009), pp.~177--194.

\bibitem{Ant2}
{\sc P.~R.~S. Antunes}, {\em On the buckling eigenvalue problem}, Journal of
  Physics A: Mathematical and Theoretical, 44 (2011), p.~215205.

\bibitem{Ant}
{\sc P.~R.~S. Antunes}, {\em Optimal {B}ilaplacian eigenvalues}, SIAM Journal
  on Control and Optimization, 52 (2014), pp.~2250--2260.

\bibitem{ABM}
{\sc M.~S. Ashbaugh, R.~Benguria, and R.~Mahadevan}, {\em A sharp lower bound
  for the first eigenvalue of the vibrating clamped plate problem under
  compression}, preprint,  (2018).

\bibitem{ashb}
{\sc M.~S. Ashbaugh and R.~D. Benguria}, {\em On {R}ayleigh's conjecture for
  the clamped plate and its generalization to three dimensions}, Duke Math. J.,
  78 (1995), pp.~1--17.

\bibitem{bt}
{\sc T.~Betcke and L.~N. Trefethen}, {\em Reviving the {M}ethod of {P}articular
  {S}olutions}, SIAM Rev., 47 (2005), pp.~469--491.

\bibitem{bucur}
{\sc D.~Bucur}, {\em Existence results. In: Shape Optimization and
  Spectral Theory, ed.\ A. Henrot}, De Gruyter Open, Warsaw/Berlin, 2017.

\bibitem{bfk}
{\sc D.~Bucur, P.~Freitas, and J.~B. Kennedy}, {\em The {R}obin problem. In:
  Shape Optimization and Spectral Theory, ed.\ A. Henrot}, De Gruyter Open,
  Warsaw/Berlin, 2017.

\bibitem{buososurvey}
{\sc D.~Buoso}, {\em Analyticity and criticality results for the eigenvalues of
  the biharmonic operator. In: Geometric properties for parabolic and elliptic
  PDE's, Springer Proc. Math. Stat., 176}, Springer, 2016.

\bibitem{bl}
{\sc D.~Buoso and P.~D. Lamberti}, {\em Eigenvalues of polyharmonic operators
  on variable domains}, ESAIM Control Optim. Calc. Var., 19 (2013),
  pp.~1225--1235.

\bibitem{chaschun}
{\sc L.~M. Chasman and J.~Chung}, {\em Spectrum of the free rod under
  tension and compression}, Applicable Anal.

\bibitem{dalmasso}
{\sc R.~Dalmasso}, {\em Un probl\`{e}me de sym\'{e}trie pour une \'{e}quation
  biharmonique}, Ann. Fac. Sci. Toulouse Math., 11 (1990), pp.~45--53.

\bibitem{dz}
{\sc M.~C. Delfour and J.-P. Zol\'{e}sio}, {\em Shapes and Geometries:
  Analysis, Differential Calculus, and Optimization}, Adv. Des. Control 4,
  SIAM, Philadelphia, 2001.

\bibitem{frank}
{\sc L.~S. Frank}, {\em Coercive singular perturbations: eigenvalue problems
  and bifurcation phenomena}, Ann. Mat. Pura Appl., 148 (1987), pp.~367--395.

\bibitem{ggs}
{\sc F.~Gazzola, H.-C. Grunau, and G.~Sweers}, {\em Polyharmonic boundary value
  problems. Positivity preserving and nonlinear higher order elliptic equations
  in bounded domains, Lecture Notes in Mathematics, 1991}, Springer-Verlag,
  Berlin, 2010.

\bibitem{Grinf}
{\sc P.~Grinfeld}, {\em Hadamard's formula inside and out}, Journal of
  Optimization Theory and Applications, 146 (2010), pp.~654--690.

\bibitem{Henrot}
{\sc A.~Henrot and M.~Pierre}, {\em Variation et optimisation de formes. Une
  analyse g\'{e}om\'{e}trique}, Springer, Series Math\'{e}matiques et
  Applications, Vol. 48, 2005.

\bibitem{rayleigh}
{\sc B.~R. J.~W.~Strutt}, {\em The Theory of Sound}, Dover Publications, New
  York, 2nd~ed., 1945.

\bibitem{kawlevvel}
{\sc B.~Kawohl, H.~A. Levine, and W.~Velte}, {\em Buckling eigenvalues for a
  clamped plate embedded in an elastic medium and related questions}, SIAM J.
  Math. Anal., 24 (1993), pp.~327--340.

\bibitem{Kita}
{\sc M.~Kitahara}, {\em Boundary integral equation methods in eigenvalue
  problems of elastodynamics and thin plates}, Elsevier, Amsterdam, 1985.

\bibitem{love}
{\sc A.~Love}, {\em A treatise on the Mathematical Theory of Elasticity}, Dover
  Publications, New York, 4th~ed., 1944.

\bibitem{nadi}
{\sc N.~S. Nadirashvili}, {\em Rayleigh's conjecture on the principal frequency
  of the clamped plate}, Arch. Rational Mech. Anal., 129 (1995), pp.~1--10.

\bibitem{nist}
{\sc F.~W.~J. Olver, D.~W. Lozier, R.~F. Boisvert, and C.~W. Clark}, eds., {\em
  NIST handbook of mathematical functions}, Cambridge University Press,
  Cambridge, 2010.

\bibitem{ortzua}
{\sc J.~Ortega and E.~Zuazua}, {\em Generic simplicity of the spectrum and
  stabilization for a plate equation}, SIAM J. Cont. Optim., 39 (2001),
  pp.~1585--1614.

\bibitem{fl}
{\sc P-Freitas and R.~Laugesen}, {\em From {N}eumann to {S}teklov via {R}obin:
  the {W}einberger way}, 2018, \url{https://arxiv.org/abs/1810.07461}.

\bibitem{pankpop}
{\sc K.~Pankrashkin and N.~Popoff}, {\em Mean curvature bounds and eigenvalues
  of robin laplacians}, Calc. Var. Partial Differential Equations, 54 (2015),
  pp.~1947--1961.

\bibitem{serrin}
{\sc J.~Serrin}, {\em A symmetry problem in potential theory}, Arch. Rational
  Mech. Anal., 43 (1971), pp.~304--318.

\end{thebibliography}

\end{document}